\documentclass[12pt]{amsart}
\pdfoutput=1

\usepackage[paperheight=11in, paperwidth=8.5in, left=1in, top=1.25in, right=1in, bottom=1.25in]{geometry}
\usepackage[linktocpage=true, linktoc=all, colorlinks=true, linkcolor=black, citecolor=black, filecolor=black, urlcolor=black, pagebackref=false, pdfstartpage={1}, pdfstartview={FitH}, pdftitle={}, pdfauthor={}, pdfsubject={}, pdfcreator={}, pdfproducer={}, pdfkeywords={}, bookmarksdepth=99]{hyperref}
\usepackage{afterpage}
\usepackage{amsfonts}
\usepackage{amsmath}
\allowdisplaybreaks[4]
\usepackage{amssymb}
\usepackage{amsthm}
\usepackage{mathtools}
\usepackage{array}
\usepackage{arydshln}
\usepackage{bbm}
\usepackage[justification=centering]{caption}
\usepackage[capitalize,nameinlink,noabbrev,compress]{cleveref}

\usepackage{enumerate}
\usepackage{enumitem}
\usepackage{float}
\restylefloat{figure}
\usepackage{graphicx}
\usepackage{kbordermatrix}
\usepackage{mathrsfs}
\usepackage[framemethod=tikz,ntheorem,xcolor]{mdframed}
\usepackage{multirow}
\usepackage{parcolumns}
\usepackage{scalerel}
\usepackage{tikz}\pdfpageattr{/Group <</S /Transparency /I true /CS /DeviceRGB>>}
\usetikzlibrary{arrows, calc, cd, decorations.markings, intersections, matrix, positioning, through}
\tikzset{commutative diagrams/.cd, every label/.append style = {font = \normalsize}}
\usepackage[all]{xy}
\usepackage[aligntableaux=center]{ytableau}

\usepackage{comment}

\numberwithin{equation}{section}
\newtheorem{theorem}{Theorem}
\AfterEndEnvironment{theorem}{\noindent\ignorespaces}
\numberwithin{theorem}{section}

\AfterEndEnvironment{conjecture}{\noindent\ignorespaces}
\newtheorem{corollary}[theorem]{Corollary}
\AfterEndEnvironment{corollary}{\noindent\ignorespaces}
\newtheorem{lemma}[theorem]{Lemma}
\AfterEndEnvironment{lemma}{\noindent\ignorespaces}
\newtheorem{problem}[theorem]{Problem}
\AfterEndEnvironment{problem}{\noindent\ignorespaces}
\newtheorem{proposition}[theorem]{Proposition}
\AfterEndEnvironment{proposition}{\noindent\ignorespaces}

\theoremstyle{definition}
\newtheorem{definition}[theorem]{Definition}
\AfterEndEnvironment{definition}{\noindent\ignorespaces}
\newtheorem{example_no_qed}[theorem]{Example}
\newenvironment{example}[1][]{\begin{example_no_qed}[#1]\pushQED{\qed}}{\popQED\end{example_no_qed}}
\AfterEndEnvironment{example}{\noindent\ignorespaces}
\newtheorem{remark}[theorem]{Remark}
\AfterEndEnvironment{remark}{\noindent\ignorespaces}

\theoremstyle{remark}

\AfterEndEnvironment{claim}{\noindent\ignorespaces}
\newtheorem*{claim_proof_no_qed}{Proof of Claim}

\AfterEndEnvironment{claim_proof}{\noindent\ignorespaces}

\AddToHook{env/conjecture/begin}{\crefalias{theorem}{conjecture}}
\AddToHook{env/corollary/begin}{\crefalias{theorem}{corollary}}
\AddToHook{env/lemma/begin}{\crefalias{theorem}{lemma}}
\AddToHook{env/problem/begin}{\crefalias{theorem}{problem}}
\AddToHook{env/proposition/begin}{\crefalias{theorem}{proposition}}
\AddToHook{env/definition/begin}{\crefalias{theorem}{definition}}
\AddToHook{env/example/begin}{\crefalias{theorem}{example}}
\AddToHook{env/example_no_qed/begin}{\crefalias{theorem}{example_no_qed}}
\AddToHook{env/remark/begin}{\crefalias{theorem}{remark}}

\crefname{theorem}{Theorem}{Theorems}
\crefname{conjecture}{Conjecture}{Conjectures}
\crefname{corollary}{Corollary}{Corollaries}
\crefname{lemma}{Lemma}{Lemmas}
\crefname{problem}{Problem}{Problems}
\crefname{proposition}{Proposition}{Propositions}
\crefname{definition}{Definition}{Definitions}
\crefname{example}{Example}{Examples}
\crefname{example_no_qed}{Example}{Examples}
\crefname{remark}{Remark}{Remarks}

\DeclareMathOperator{\cotype}{cotype}

\renewcommand{\eqref}[1]{\hyperref[#1]{\textup{(\ref*{#1})}}}

\DeclareMathOperator{\Fl}{Fl}

\DeclareMathOperator{\Gr}{Gr}

\DeclareMathOperator{\Hom}{Hom}

\DeclareMathOperator{\id}{id}

\DeclareMathOperator{\SL}{SL}

\DeclareMathOperator{\SO}{SO}

\DeclareMathOperator{\Sp}{Sp}

\DeclareMathOperator{\Stab}{Stab}

\newcommand{\transpose}[1]{{#1}{\hspace*{-0.2pt}\raisebox{4pt}{$\scriptstyle\mathsf{T}$}\hspace*{0.5pt}}}

\DeclareMathOperator{\type}{type}

\newcommand{\CC}{\mathbb{C}}

\newcommand{\PP}{\mathbb{P}}

\newcommand{\RR}{\mathbb{R}}

\newcommand{\cO}{\mathcal{O}}
\newcommand{\cP}{\mathcal{P}}
\newcommand{\cQ}{\mathcal{Q}}

\newcommand{\cS}{\mathcal{S}}

\newcommand{\affine}{\mathrm{affine}}

\newcommand{\coset}{\mathrm{coset}}

\DeclareMathOperator{\LG}{LG}
\newcommand{\newword}[1]{{\bfseries\itshape #1}}
\DeclareMathOperator{\OG}{OG}

\newcommand{\standard}{\mathrm{standard}}

\newcommand{\symgp}[1]{\mathfrak{S}_{#1}}

\title{From amplituhedra to flagtopes}

\author{Grant T. Barkley}
\address{Department of Mathematics, University of Michigan}
\email{\href{mailto:gbarkley@umich.edu}{gbarkley@umich.edu}}
\author{Steven N. Karp}
\address{Department of Mathematics, University of Notre Dame}
\email{\href{mailto:skarp2@nd.edu}{skarp2@nd.edu}}
\author{Kaitao Xie}
\address{Department of Mathematics, The University of Hong Kong}
\email{\href{mailto:kaitaoxie@connect.hku.hk}{kaitaoxie@connect.hku.hk}}


\begin{document}

\begin{abstract}
The amplituhedron was introduced by Arkani-Hamed and Trnka (2014) to geometrically encode scattering amplitudes in high-energy physics. It can be viewed as the generalization of a cyclic polytope into the Grassmannian, and is defined to be the projection of the totally nonnegative part of one Grassmannian with respect to a totally negative element of a second Grassmannian. We construct a Lie-theoretic generalization of amplituhedra called \emph{flag amplituhedra}, by replacing the pair of Grassmannians in the definition with an arbitrary pair of partial flag varieties for a reductive group $G$. We study in detail an example of a flag amplituhedron which is not a Grassmannian amplituhedron; it is a curved hexagon in $\mathbb{P}^1\times \mathbb{P}^1$.
\end{abstract}

\maketitle

\section{Introduction}\label{sec:introduction}

\noindent The \newword{amplituhedron} $\mathcal{A}_{n,k,m}(Z)$ was introduced by Arkani-Hamed and Trnka \cite{arkani-hamed_trnka14} in order to give a geometric basis for computing scattering amplitudes in physics. It has since been intensively studied in both physics and combinatorics; see \cref{sec:problems} and the survey \cite{williams23} for more detailed discussion. The purpose of this paper is to formulate a Lie-theoretic generalization of amplituhedra, which we propose as an interesting object for further study.

In order to explain our construction, we recall the definition of the amplituhedron. It depends on $Z$, which is a $(k+m)\times n$ real matrix whose maximal minors are all positive. We view $Z$ as a linear map from $\mathbb{C}^n$ to $\mathbb{C}^{k+m}$, which induces a rational map $Z_{\Gr} : \Gr_{k,n}(\CC) \dashrightarrow \Gr_{k,k+m}(\CC)$, where $\Gr_{k,m}(\CC)$ is the set of all $k$-dimensional linear subspaces of $\CC^m$. The map $Z_{\Gr}$ is well-defined on the totally nonnegative Grassmannian $\Gr_{k,n}^{\ge 0}$ (the subset where all Pl\"{u}cker coordinates are nonnegative), and the amplituhedron is defined to be its image:
\[
\mathcal{A}_{n,k,m}(Z) \coloneqq Z_{\Gr}(\Gr_{k,n}^{\ge 0}) \subseteq \Gr_{k,k+m}(\CC).
\]
When $k=1$, the amplituhedron can be realized as a cyclic polytope in $\mathbb{R}^m$ with $n$ vertices \cite{sturmfels88}; when $m=1$, it can be realized as the bounded complex of a cyclic hyperplane arrangement in $\mathbb{R}^k$ with $n$ hyperplanes \cite{karp_williams19} (see \cref{figure_amplituhedra} for examples).

Following \cite{karp_williams19}, we can equivalently formulate the amplituhedron by working with the map
\[
\rho_{\ker(Z)} : \Gr_{k,n}(\CC) \dashrightarrow \Gr_{n-m,n}(\CC), \enspace V\mapsto V + \ker(Z).
\]
This only depends on $\ker(Z)$ (rather than $Z$), which is an element of the totally negative Grassmannian $\Gr_{n-k-m,n}^{<0}$. We then have
\[
\mathcal{A}_{n,k,m}(Z) \simeq \rho_{\ker(Z)}(\Gr_{k,n}^{\ge 0}),
\]
where the image lies in the Grassmannian $\Gr_k(\mathbb{C}^n/\ker(Z))$ of all $k$-dimensional subspaces of $\mathbb{C}^n/\ker(Z)$.

Our main contribution generalizes this construction by replacing the two Grassmannians $\Gr_{n-k-m,n}(\CC)$ and $\Gr_{k,n}(\CC)$ above with two arbitrary partial flag varieties of a reductive group $G$. Explicitly, given two subsets $J$ and $J'$ of the set of simple roots $I$, we replace $\Gr_{n-k-m,n}(\CC)$ with $\cP^J(G)$, the flag variety of all parabolic subgroups of $G$ conjugate to $P_J^+$; and we replace $\Gr_{k,n}(\CC)$ with $\cQ^{J'}(G)$, the flag variety of all parabolic subgroups conjugate to $P_{J'}^-$ (all notation is defined in \cref{sec:flagtopes}). Given $P\in\cP^J(G)$, we also replace $\rho_{\ker(Z)}$ with a rational map $\rho_P : \cQ^{J'}(G) \dashrightarrow \cQ^{J\cap J'}(L)$, where $L = P/U_P$ is the Levi factor of $P$. This map essentially appeared in work of Borel and Tits \cite{borel_tits65}; we call it \newword{$P$-restriction}. Also, $\cQ^{J'}(G)$ has a totally nonnegative part $\cQ^{J'}_{\ge 0}(G)$ defined by Lusztig \cite{lusztig98}, and similarly $\cP^J(G)$ has a totally negative part $\cP^J_{<0}(G)$. Then our generalization of amplituhedra is as follows:
\begin{figure}[tb]
\begin{center}
\begin{align*}
\arraycolsep=24pt
\begin{array}{cc}
\includegraphics[scale=0.7]{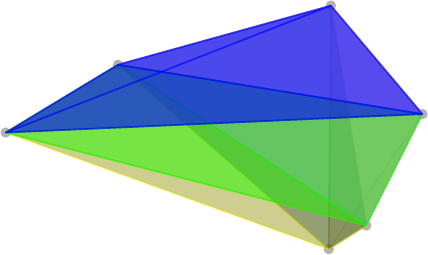} & \includegraphics[scale=0.9]{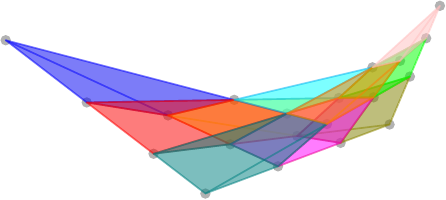} \\
\hspace*{11pt}\mathcal{A}_{6,1,3}(Z) & \hspace*{-7pt}\mathcal{A}_{6,3,1}(Z)
\end{array}
\end{align*}
\caption{When $k=1$, every amplituhedron is a cyclic polytope \cite{sturmfels88}. When $m=1$, we can identify any amplituhedron with the complex of bounded faces of a cyclic hyperplane arrangement \cite{karp_williams19}.}
\label{figure_amplituhedra}
\end{center}
\end{figure}
\begin{theorem}\label{theorem_intro}
Let $P\in\mathcal P^J_{<0}(G)$ with Levi factor $L$. Then the $P$-restriction map is well-defined on $\cQ^{J'}_{\ge 0}(G)$, and we call its image $\rho_P(\cQ^{J'}_{\ge 0}(G))$ in $\cQ^{J\cap J'}(L)$ a \newword{flag amplituhedron}. When $G = \SL_n(\CC)$ and $\cP^J(G)$ and $\cQ^{J'}(G)$ are both Grassmannians, flag amplituhedra specialize to the usual amplituhedra introduced by Arkani--Hamed and Trnka.
\end{theorem}
\begin{figure}[tb]
\begin{center}
\[
\begin{tikzpicture}[baseline=(current bounding box.center),scale=2]
\tikzstyle{out1}=[inner sep=0,minimum size=1.2mm,circle,draw=black,fill=black,semithick]
\pgfmathsetmacro{\s}{0.9};
\pgfmathsetmacro{\xmin}{-0.5};
\pgfmathsetmacro{\xmax}{3.5};
\pgfmathsetmacro{\ymin}{-0.5};
\pgfmathsetmacro{\ymax}{3.5};
\draw[help lines,color=black!40,dashed](\xmin,\ymin)grid(\xmax,\ymax);
\draw[latex'-latex',very thick](\xmin,0)--(\xmax,0);
\draw[latex'-latex',very thick](0,\ymin)--(0,\ymax);
\node[inner sep=0]at(3,0)[label={[below=1pt]$3$}]{};
\node[inner sep=0]at(0,3)[label={[left=0pt]$3$}]{};
\node[inner sep=0]at(\xmax,0)[label={[right=-1pt]$x$}]{};
\node[inner sep=0]at(0,\ymax)[label={[above=-2pt]$y$}]{};
\draw[thick,fill=teal!30](0,0)--(2,0)--plot[smooth,samples=100,domain=2:3](\x,{(6-3*\x)/(3-2*\x)})--(3,1)--(3,3)--(1,3)--plot[smooth,samples=100,domain=1:0](\x,{(6-3*\x)/(3-2*\x)})--(0,2)--cycle;
\node[out1](v1234)at(3,3)[label={[above right=-4pt]\scalebox{\s}{$1234$}}]{};
\node[out1](v2134)at(3,2)[label={[right=-1pt]\scalebox{\s}{$2134$}}]{};
\node[out1](v1243)at(2,3)[label={[above=-2pt]\scalebox{\s}{$1243$}}]{};
\node[out1](v3124)at(3,1)[label={[right=-1pt]\scalebox{\s}{$3124$}}]{};
\node[out1](v2143)at(2,2)[label={[above right=-4pt]\scalebox{\s}{$2143$}}]{};
\node[out1](v1342)at(1,3)[label={[above=-2pt]\scalebox{\s}{$1342$}}]{};
\node[out1](v4123)at(2,0)[label={[below=3pt]\scalebox{\s}{$4123$}}]{};
\node[out1](v3142)at(1,1)[label={[below left=2pt]\scalebox{\s}{$3142\!\!$}}]{};
\node[out1](v2341)at(0,2)[label={[left=-1pt]\scalebox{\s}{$2341$}}]{};
\node[out1](v4132)at(1,0)[label={[below=3pt]\scalebox{\s}{$4132$}}]{};
\node[out1](v3241)at(0,1)[label={[left=-1pt]\scalebox{\s}{$3241$}}]{};
\node[out1](v4231)at(0,0)[label={[below left=3pt]\scalebox{\s}{$4231\!\!\!$}}]{};
\end{tikzpicture}
\]
\caption{The flag amplituhedron $\rho_V(\Fl_{1,3;4}^{\ge 0})$ from \cref{sec:hexagon}, where $V\in\Gr_{2,4}^{<0}$. The two curved boundaries are part of the two branches of the hyperbola $(2x-3)(2y-3)=-3$. We have also depicted the images of the permutation flags in $\Fl_{1,3;4}^{\ge 0}$, which are located at lattice points.}
\label{figure_hexagon}
\end{center}
\end{figure}
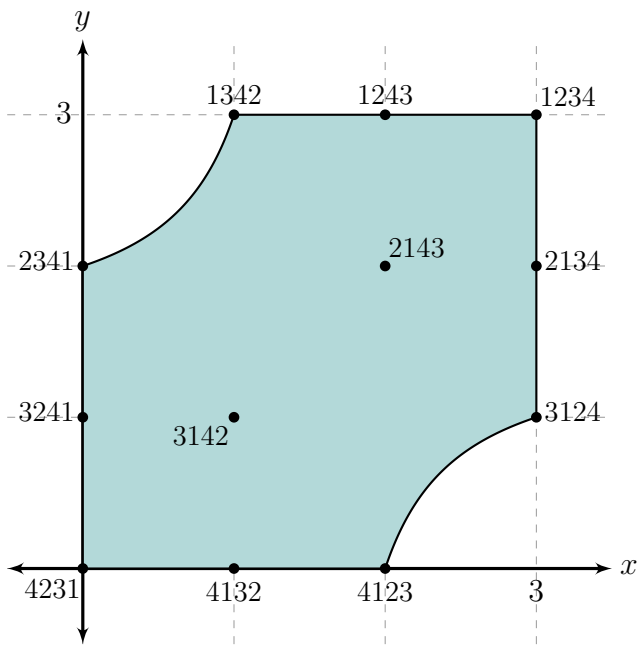

Just as amplituhedra are generalizations of cyclic polytopes inside Grassmannians, Lam \cite{lam16} introduced \newword{Grassmann polytopes} as generalizations of arbitrary polytopes into Grassmannians. (Every amplituhedron is a Grassmann polytope.) We similarly introduce \newword{flagtopes} in \cref{flagtope_definition}, which generalizes Lam's construction from a pair of Grassmannians to an arbitrary pair of partial flag varieties of $G$.

We have two main motivations for studying flag amplituhedra and flagtopes:
\begin{enumerate}[label=(\arabic*), leftmargin=36pt, itemsep=4pt]
\item Flagtopes provide new examples of spaces which may be of interest in physics, combinatorics, and positive geometries. For example, \cref{figure_hexagon} exhibits an example of a flag amplituhedron (which is not a classical amplituhedron studied previously) which is a curved hexagon and a positive geometry.
\item Adopting a more general Lie-theoretic framework may help solve problems about amplituhedra (such as finding their tilings and determining their topology) even in the case of Grassmannian amplituhedra, since it provides a different perspective and allows one to leverage tools from representation theory and Coxeter groups.
\end{enumerate}

\subsection*{Outline}
In \cref{sec:flagtopes} we establish the necessary background on partial flag varieties and $P$-restriction, culminating in the definition of flagtopes and flag amplituhedra in \cref{sec:flagtopes_definition}. (We prove \cref{theorem_intro} in \cref{flagtope_definition} and \cref{sec:examples_amplituhedra}.) In \cref{sec:type_A} we give an explicit description of $P$-restriction in type $A$ (when $G = \SL_n(\CC)$). In \cref{sec:examples} we show that several spaces appearing in the literature (namely tree and one-loop amplituhedra, Grassmann polytopes, ABJM amplituhedra, and Lagrangian amplituhedra) are all special cases of flagtopes. In \cref{sec:problems} we present several open problems about flagtopes, motivated by work on amplituhedra. Finally, in \cref{sec:hexagon}, we study in detail the example of a flag amplituhedron from \cref{figure_hexagon}.

\subsection*{AI usage}
The authors formulated the constructions, results, and proofs without AI assistance. The authors used ChatGPT to derive the inequalities \eqref{hexagon_equations} and the matrix $h$ in \Cref{affineexample} in the curved hexagon example, which were then confirmed by the authors (the rest of the example was done by hand and with \texttt{Maple}), as well as for finding references on $P$-restriction and for proofreading a completed draft of the paper.

\subsection*{Acknowledgments}
G.T.B.\ was supported by the National Science Foundation under Award No.\ 2503536. S.N.K.\ was partially supported by the National Science Foundation under Award No.\ 2452061 and by a travel support gift from the Simons Foundation. K.X.\ was partially supported by NSFC grant 125B2002 and the New Cornerstone Science Foundation through the New Cornerstone Investigator Program awarded to Professor Xuhua He. G.T.B.\ and K.X.\ thank University of Notre Dame for hospitality during the Thematic Program: Algebraic Combinatorics and Applications.

\section{Defining flagtopes}\label{sec:flagtopes}

\noindent In this section, we recall some background on reductive groups, parabolic subgroups, and total positivity; see \cite{borel_tits65,humphreys75,lusztig94,barkley_karp} for further details.

\subsection{Reductive groups and flag varieties}
Let $G$ be a complex connected reductive group which is split over $\RR$. We let $\cdot$ denote the action of $G$ by conjugation, so for example $g\cdot H = \{ghg^{-1} : h\in H\}$ for all $H\subseteq G$. Fix an $\RR$-split maximal torus $T\subseteq G$ and a pair of opposite Borel subgroups $B^+$ and $B^-$ such that $B^+\cap B^- = T$. Let $X(T) \coloneqq \Hom(T,\CC^\times)$ denote the \newword{weight lattice}, and let $\Phi \subseteq X(T)$ denote the set of \newword{roots}. Our choice of $B^+$ and $B^-$ induces the decomposition $\Phi = \Phi^+ \sqcup \Phi^-$, where $\Phi^+$ is the set of \newword{positive roots} and $\Phi^-$ is the set of \newword{negative roots}. Let $\{\alpha_i\mid i\in I\} \subseteq \Phi^+$ denote the set of \newword{simple roots}, where $I$ is an indexing set. Let $U^+$ and $U^-$ be the unipotent radicals of $B^+$ and $B^-$, respectively. Each positive root $\alpha$ gives rise to one-parameter \newword{root subgroups} $U_\alpha \subseteq U^+$ and $U_{-\alpha}\subseteq U^-$.

For each $i\in I$, we pick a homomorphism $\phi_i: \SL_2(\CC) \to G$, and define
\begin{align}\label{phi_pinning}
x_i(t) \coloneqq \phi_i\!\left(\begin{bmatrix}1 & t \\ 0 & 1\end{bmatrix}\right),
\enspace
y_i(t) \coloneqq \phi_i\!\left(\begin{bmatrix}1 & 0 \\ t & 1\end{bmatrix}\right),
\enspace
\dot{s}_i \coloneqq \phi_i\!\left(\begin{bmatrix}0 & -1 \\ 1 & 0\end{bmatrix}\right).
\end{align}
We require our choice of $\phi_i$ to be defined over $\RR$ and to satisfy $x_i(\CC) = U_{\alpha_i}$ and $y_i(\CC) = U_{-\alpha_i}$. The data $(T,B^+,B^-,\{x_i \mid i\in I\},\{y_i \mid i\in I\})$ is called a \newword{pinning} for $G$.

We let $W \coloneqq N(T)/T$ denote the \newword{Weyl group} of $G$. Let $s_i \in W$ denote the element of $W$ represented by $\dot s_i$, so that $W$ is a finite Coxeter group with simple generators $\{s_i\mid i\in I\}$ (see \cite{bjorner_brenti05} for details). We let $e$ denote the identity element of $W$. A \newword{reduced expression} for $w\in W$ is a minimal-length expression $w = s_{i_1} \cdots s_{i_l}$ as a product of simple generators. We define $\dot{w} \coloneqq \dot{s}_{i_1} \cdots \dot{s}_{i_l}$, which is a lift of $w$ in $N(T)$ which does not depend on the choice of reduced expression for $w$. We let $w_0$ denote the longest element of $W$. We let $\tau_I$ denote the standard involution on $I$ (often denoted $\cdot^*$ in the literature), which is defined by
\[
w_0(\alpha_i) = -\alpha_{\tau_I(i)} \enspace \text{ for all } i\in I.
\]

Similarly, for $J\subseteq I$, we let $W_J$ denote the subgroup of $W$ generated by $s_j$ for all $j\in J$, and we denote its longest element by $w_{0,J}$. Then we let $\tau_J$ be the involution of $J$ defined by $w_{0,J}(\alpha_i) = - \alpha_{\tau_J(i)}$ for $i\in J$.

\begin{example}\label{eg:pinning}
We recall the standard pinning for $G = \SL_n(\CC)$, the group of invertible $n\times n$ matrices with determinant $1$. We let $B^+$ be the subgroup of upper-triangular matrices and $B^-$ be the subgroup of lower-triangular matrices, so that $T = B^+ \cap B^-$ is the subgroup of diagonal matrices. We have $I = [n-1] \coloneqq \{1, \dots, n-1\}$, and for $i\in I$ the map $\phi_i$ embeds $\SL_2(\CC)$ in rows and columns $\{i,i+1\}$. For example,
\[
x_i(t) = \kbordermatrix{& & i & i+1 & \cr
& \ddots & & & \cr
i & & 1 & t & \cr
i+1 & & 0 & 1 & \cr
& & & & \ddots} \enspace \text{ and } \enspace y_i(t) = \kbordermatrix{& & i & i+1 & \cr
& \ddots & & & \cr
i & & 1 & 0 & \cr
i+1 & & t & 1 & \cr
& & & & \ddots},
\]
where the matrices agree with the identity matrix $I_n$ in all unspecified entries. The Weyl group $W$ is the symmetric group $\symgp{n}$ of all permutations of $[n]$. For example, if $w = 312\in\symgp{3}$ then
\[
\dot{w} = \begin{bmatrix}
0 & -1 & 0 \\
0 & 0 & -1 \\
1 & 0 & 0
\end{bmatrix} \in\SL_3(\CC).
\]
The permutation $w_0$ is the involution sending $i$ to $n+1-i$ for all $i\in [n]$. For all $i\in [n-1]$, the simple generator $s_i$ is the transposition swapping $i$ and $i+1$, and $\tau_I(i) = n-i$.
\end{example}

A \newword{parabolic subgroup} is a closed and connected subgroup $P$ of $G$ that contains a Borel subgroup. We write $U_P$ for the unipotent radical of $P$. The group $L\coloneqq P/U_P$ is called the \newword{Levi factor} of $P$, which is also a reductive group.

For $J\subseteq I$, let $P_J^+$ be the standard parabolic subgroup generated by $B^+$ and $U_{-\alpha_i}$ for all $i\in J$. Similarly, the opposite parabolic subgroup $P_J^-$ is the parabolic subgroup generated by $B^-$ and $U_{\alpha_i}$ for all $i\in J$. In particular, $P_\emptyset^\pm = B^\pm$, and $P^\pm_I = G$.

Let $L_J \coloneqq P_J^+\cap P_J^-$ be the standard Levi subgroup. We have the Levi decomposition $P_J^\pm = L_J\ltimes U_{P_J^\pm}$, and thus the group isomorphism $L_J\simeq P_J^+/U_{P_J^+}$. We let $\Phi_J$ denote the root system of $L_J$; also note that the Weyl group of $L_J$ is $W_J$.

\subsection{Types and cotypes}\label{sec:types}
For every parabolic subgroup $P\subseteq G$, there exist $g\in G$ and a unique $J\subseteq I$ such that $g\cdot P = P_J^+$. We call $J$ the \newword{type} of $P$, denoted $\type_G(P)$. Similarly, we define the \newword{cotype} of $P$ to be the unique $J'\subseteq I$ such that $P$ is conjugate to $P_{J'}^-$, denoted $\cotype_G(P)$.

\begin{remark}\label{cotype_motivation}
Since $\dot{w}_0\cdot P_J^+ = P_{\tau_I(J)}^-$ we have
\begin{align}\label{type_cotype}
\cotype_G(P) = \tau_I(\type_G(P)),
\end{align}
i.e., the involution $\tau_I$ exchanges types and cotypes. Therefore one can work exclusively with types and avoid cotypes by using \eqref{type_cotype}. However, we prefer to work with both types and cotypes for two reasons:
\begin{itemize}[itemsep=2pt]
\item Certain results, such as \eqref{restriction_cotype_formula}, are particularly elegant and easier to understand when stated using both types and cotypes.
\item Much of the setup of this section generalizes to Kac--Moody groups, where cotypes are well-defined but we no longer have a formula like \eqref{type_cotype} relating types and cotypes. (To avoid making this paper overly technical we restrict ourselves to the case of reductive groups, and leave open the study of flagtopes for Kac--Moody groups.)
\end{itemize}
\end{remark}

Given $J\subseteq I$, let $\mathcal P^J$ be the space of all parabolic subgroups of type $J$, and let $\mathcal Q^J$ be the space of all parabolic subgroups of cotype $J$. Since the normalizer of $P_J^+$ in $G$ is $P_J^+$ itself, we have the isomorphism
\[
\mathcal P^J\xrightarrow[]{\sim}G/P_J^+, \enspace g\cdot P_J^+\mapsto gP_J^+/P_J^+.
\]
Since $P_J^+$ is a parabolic subgroup, $\mathcal P^J$ is endowed with the structure of a projective homogeneous $G$-variety, and it is called a \newword{(partial) flag variety} of $G$.
Similarly, there is an isomorphism
\[
\mathcal Q^J \xrightarrow\sim G/P_J^-, \enspace g\cdot P_J^- \mapsto gP_J^-/P_J^-.
\]
\begin{remark}\label{perp_remark}
Let $\transpose{\cdot}$ denote the involutive anti-automorphism of $G$ which fixes every element of $T$ and takes $x_i(t)$ to $y_i(t)$ for all $i\in I$. We can check that the map
\begin{align}\label{perp_formula}
\cdot^\perp : \mathcal{P}^J \to \mathcal{Q}^J, \enspace g\cdot P_J^+ \mapsto \transpose{(g^{-1})}\cdot P_J^-
\end{align}
is well-defined and an isomorphism ($\cdot^\perp$ was defined in \cite[(2.4)]{barkley_karp} for the special case of Borel subgroups). This naturally associates a parabolic subgroup of type $J$ to one of cotype $J$ in such a way that $P_J^+$ corresponds to $P_J^-$. Moreover, $\cdot^\perp$ takes a totally positive parabolic subgroup to a totally negative parabolic subgroup (these notions will be defined in \cref{sec:total_positivity}).
\end{remark}

For a parabolic subgroup $P$, its Levi factor $L$ is also a reductive group. A parabolic subgroup of $L$ itself has a type and cotype, which we define via the following lemma:
\begin{lemma}\label{type_L}
Let $P\in\cP^J$ with Levi factor $L = P/U_P$, and let $Q$ be a parabolic subgroup of $L$. Then there exists a unique parabolic subgroup $Q'$ of $G$ such that $U_P \subseteq Q'\subseteq P$ and $Q'/U_P = Q$. Moreover, $\type_G(Q') \subseteq J$ only depends on the $L$-conjugacy class of $Q$.
\end{lemma}

\begin{proof}
Let $\pi:P\to L$ be the natural projection, and consider $Q' = \pi^{-1}(Q)$. We have $U_P\subseteq Q'\subseteq P$, and $Q'$ is unique. Let $B_L$ be a Borel subgroup of $L$ that is contained in $Q$, so that $\pi^{-1}(B_L)$ is a Borel subgroup of $G$ and is contained in $Q'$. Therefore $Q'$ is a parabolic subgroup of $G$. Since $Q'\subseteq P$, we have $\type_G(Q')\subseteq J$.

Now suppose that $Q_1$ and $Q_2$ are $L$-conjugate parabolic subgroups of $L$. By definition, there exists $p\in P$ such that \[
p\pi^{-1}(Q_1)p^{-1}U_P/U_P = \pi^{-1}(Q_2)U_P/U_P.
\]
We conclude that $p\pi^{-1}(Q_1)p^{-1} = \pi^{-1}(Q_2)$. Therefore, $\type_G(Q')$ only depends on the $L$-conjugacy class of $Q$.
\end{proof}

Given $P\in\mathcal{P}^J$ with Levi factor $L$ and a parabolic subgroup $Q$ of $L$, we define
\[
\type_L(Q) \coloneqq \type_G(Q') \enspace \text {for $Q'$ as in \cref{type_L}}.
\]
Equivalently, we can define type and cotype in $L$ by assuming that $P = P_J^+$ and taking the standard isomorphism $L \simeq L_J$; we then view the simple reflections of $L$ as a subset of the simple reflections of $G$ using $L\simeq L_J \subseteq G$. Then by \eqref{type_cotype}, we have
\[
\cotype_L(Q) = \tau_J(\type_L(Q)) \subseteq J.
\]

\begin{remark}\label{cotype_asymmetry}
    Every reductive group $G$ has an `intrinsic' notion of type/cotype, which is a subset of the indexing set of all conjugacy classes of maximal parabolic subgroups of $G$. What we have implicitly done to define $\type_L$ is to associate every conjugacy class for $L$ to a conjugacy class for $G$. This association, however, is not intrinsic, and in particular it breaks a symmetry between types and cotypes. For example, $\cotype_L(Q)$ is generally \emph{not} equal to $\cotype_G(Q')$ for $Q'$ as in \cref{type_L}. (Indeed, $\cotype_G(Q')$ is a subset of $\tau_I(J)$, but it is not necessarily contained in $J$.)
\end{remark}

\subsection{Transversality for parabolic subgroups}
Let $J,J'\subseteq I$, and consider the $G$-action on $\mathcal P^J\times\mathcal Q^{J'}$ via
\[
g\cdot (P,P') \coloneqq (g\cdot P,g\cdot P').
\]
The $G$-orbits are indexed by the double quotient $W_J\backslash W/W_{J'}$, where $w\in W_J\backslash W/W_{J'}$ labels the $G$-orbit through $(P_J^+,\dot{w}\cdot P_{J'}^-)$. There is a unique open orbit, denoted by $\mathcal O_{J,J'}$, which is indexed by $e\in W_J\backslash W/W_{J'}$. We say that parabolic subgroups $P$ and $P'$ of $G$ are \newword{transverse} if $(P,P')\in\mathcal{O}_{J,J'}$ (where $J = \type_G(P)$ and $J' = \cotype_G(P')$), i.e., there exists $g\in G$ such that $g\cdot P = P_J^+$ and $g\cdot P' = P_{J'}^-$. (Transversality is a symmetric relation.) We have the following equivalent characterizations of transversality:
\begin{proposition}[{\cite[Proposition 4.3]{barkley_karp}}]\label{prop:transverseparabolics}
Let $P$ and $P'$ be parabolic subgroups of $G$. The following are equivalent:
\begin{enumerate}[label=(\roman*), leftmargin=*, itemsep=2pt]
\item\label{prop:transverseparabolics_transverse} $P$ and $P'$ are transverse;
\item\label{prop:transverseparabolics_borels} there exist Borel subgroups $B, B' \subseteq G$ such that $B\subseteq P$, $B'\subseteq P'$, and $B\cap B'$ is a maximal torus of $G$;
\item\label{prop:transverseparabolics_roots2} every root $\alpha\in\Phi$ (with respect to a fixed maximal torus in $P\cap P'$) is a root of $P$ or $P'$; 
\item\label{prop:transverseparabolics_dimension} $\dim(P\cap P') = \dim(P) + \dim(P') - \dim(G)$.\hfill\qed
\end{enumerate}
\end{proposition}

We can also characterize transversality for arbitrary parabolic subgroups in terms of transversality for maximal parabolic subgroups, as follows (see \cref{type_A_transversality} for this statement in the case of $G = \SL_n(\CC)$). We point out that given a parabolic subgroup $P\subseteq G$, the maximal parabolic subgroups containing $P$ are indexed by the subsets of $I$ of size $|I|-1$ containing $\type_G(P)$; this follows from \cite[Section 23.1]{humphreys75}.
\begin{lemma}\label{transversality_maximal_parabolics}
Let $P$ and $P'$ be parabolic subgroups of $G$. Then $P$ and $P'$ are transverse if and only if $Q$ and $Q'$ are transverse for all maximal parabolic subgroups $Q\supseteq P$ and $Q'\supseteq P'$.
\end{lemma}

\begin{proof}
    The forward direction follows from the equivalence \ref{prop:transverseparabolics_transverse} $\Leftrightarrow$ \ref {prop:transverseparabolics_borels} of \cref{prop:transverseparabolics}. To prove the backward direction, we can assume that $T \subseteq P\cap P'$. By the implication \ref{prop:transverseparabolics_roots2} $\Rightarrow$ \ref{prop:transverseparabolics_transverse} of \cref{prop:transverseparabolics}, it is enough to prove that every root $\alpha \in \Phi$ is a root of $P$ or $P'$.
    
    Assume $\alpha$ is not a root of $P$; we will show that $\alpha$ is a root of $P'$. Since $P'$ is the intersection of the maximal parabolic subgroups containing it, it suffices to show that $\alpha$ is a root of $Q'$ for every maximal parabolic subgroup $Q'\supseteq P'$. Similarly $P$ is the intersection of the maximal parabolic subgroups containing it, so there exists a maximal parabolic subgroup $Q\supseteq P$ such that $\alpha$ is not a root of $Q$. Since $Q$ and $Q'$ are transverse, $\alpha$ is a root of $Q'$ by the implication \ref{prop:transverseparabolics_transverse} $\Rightarrow$ \ref{prop:transverseparabolics_roots2} of \cref{prop:transverseparabolics}, as desired.
\end{proof}

\subsection{\texorpdfstring{$P$}{P}-restriction}
Given parabolic subgroups $P,P'\subseteq G$, we consider $(P\cap P')U_P$, which is a parabolic subgroup of $G$ by work of Borel and Tits \cite[Proposition 4.4(b)]{borel_tits65}. It is denoted $P^{P'}$ by Lusztig \cite{lusztig04}; we also refer to work of Moonen and Wedhorn \cite[Section 3.7]{moonen_wedhorn04}, where it is called the \emph{refinement} of $P$ with respect to $P'$.

\begin{lemma}\label{restriction_cotype}
Let $P$ and $P'$ be transverse parabolic subgroups of $G$, and let $L=P/U_P$ be the Levi factor of $P$. Then $(P\cap P')U_P/U_P$ is a parabolic subgroup of $L$, and
\begin{align}\label{restriction_cotype_formula}
\cotype_L((P\cap P')U_P/U_P) = \type_G(P) \cap \cotype_G(P').
\end{align}
\end{lemma}

\begin{proof}
Let $J = \type_G(P)$ and $J' = \cotype_G(P')$. We first assume that $P = P_J^+$ and $P' = P_{J'}^-$. In this case, since $(P_J^+\cap P_{J'}^-)U_{P_J^+}$ is closed and contains $T$, it is determined by which root subgroups it contains. The roots which appear are exactly
\[ \Phi_J^- \cup \Phi_{J'}^+ \cup (\Phi^+\setminus \Phi_J^+). \]
Hence the roots contained in $\Phi_J$ are exactly 
\[ \Phi_J^- \cup \Phi_{J\cap J'}^+. \]
It follows that 
\[ (P_{J}^+ \cap P_{J'}^-)U_{P_J^+}/U_{P_J^+}  \in \cQ^{J\cap J'}(L). \]

In general, there exists $g\in G$ such that $P = g\cdot P_J^+$ and $P' = g\cdot P_{J'}^-$, and
\[
(P\cap P')U_P/U_P = (g\cdot(P_J^+\cap P_{J'}^-)U_{P_J^+})/(g\cdot U_{P^+_J}) \in\mathcal{Q}^{J\cap J'}(L).
\]
This completes the proof.
\end{proof}

We can restate \eqref{restriction_cotype_formula} purely in terms of types, as follows:
\begin{corollary}\label{restriction_type}
Let $P$ and $P'$ be transverse parabolic subgroups of types $J$ and $\tilde{J}$, respectively, and let $L = P/U_P$ be the Levi factor of $P$. Then
\begin{align}\label{type_formula}
\type_L((P\cap P')U_P/U_P) =  \tau_J(J\cap \tau_I(\tilde{J})).
\end{align}
\end{corollary}

\begin{proof}
This follows from \eqref{restriction_cotype_formula} by applying \eqref{type_cotype} twice, using the fact that $\tau_J(J) = J$.
\end{proof}

Fix $J,J'\subseteq I$ and $P\in\mathcal{P}^J$. We let
\[
\mathcal O_{J,J'}^P \coloneqq \{P'\in\mathcal{Q}^{J'} \mid (P,P')\in\mathcal O_{J,J'}\}
\]
denote the set of parabolic subgroups of cotype $J'$ which are transverse to $P$.
\begin{definition}
Let $J,J'\subseteq I$, $P\in\mathcal{P}^J$, and $L = P/U_P$. We define \newword{$P$-restriction} to be the map
\begin{align}\label{restriction_definition}
\rho_P:\mathcal O_{J,J'}^P\to\mathcal{Q}^{J\cap J'}(L),\enspace P'\mapsto (P\cap P')U_P/U_P,
\end{align}
which is well-defined by \cref{restriction_cotype}. (We note that $\rho_P$ depends on $J'$ in addition to $P$, which we suppress from the notation.)
\end{definition}

Recall that $\mathcal O_{J,J'}$ is a $G$-orbit of $\cP^J \times \cQ^{J'}$. Since the normalizer of $P$ in $G$ is $P$ itself, we see that $\mathcal O_{J,J'}^P$ carries a transitive $P$-action via conjugation. The flag variety $\mathcal{Q}^{J\cap J'}(L)$ is a homogeneous $L$-space on which $P$ also acts transitively via the natural projection $P\twoheadrightarrow L$. The map $\rho_P$ is then $P$-equivariant, and therefore a regular morphism between algebraic varieties.

When $J=I$ or $J \subseteq J'$, the $P$-restriction map becomes easy to describe:
\begin{example}\label{restriction_examples}
Let $J,J'\subseteq I$, and let $K = J\cap J'$ denote the cotype of the $P$-restriction image.
\begin{enumerate}[label=(\arabic*), leftmargin=36pt, itemsep=2pt]
    \item\label{restriction_examples_identity} Suppose that $J = I$. Then $P = L = G$, $\mathcal O_{J,J'}^P = \mathcal Q^{J'}$, and $K = J'$. The $P$-restriction $\rho_P$ is the identity morphism of $\mathcal Q^{J'}$.

    \item\label{restriction_examples_singleton} Suppose that $J\subseteq J'$. Then $K = J$ and the codomain $\cQ^K(L) = \{L\}$ is a singleton.\qedhere
\end{enumerate}
\end{example}

\begin{remark}\label{restriction_other_relative_positions}
One can generalize $P$-restriction by replacing $\mathcal{O}_{J,J'}$ by an arbitrary $G$-orbit of $\mathcal{P}^J\times\mathcal{Q}^{J'}$, which is indexed by a double coset $w\in W_J\backslash W/W_{J'}$. The cotype of the image (as a parabolic subgroup of $L$) will then depend on $w$. For simplicity we only consider $P$-restriction on the dense orbit $\mathcal{O}_{J,J'}$.
\end{remark}

\subsection{Total positivity}\label{sec:total_positivity}
Following \cite{lusztig94}, we introduce the \newword{totally nonnegative} and \newword{totally positive parts} of various spaces introduced above, which we denote by adding `$\ge\!0$' and `$>\!0$', respectively, as a superscript or subscript. Let $U^+_{\ge0}$ be the submonoid of $G$ generated by $x_i(t)$ for all $i\in I$ and $t\in\mathbb R_{>0}$, and $U^-_{\ge0}$ be the submonoid of $G$ generated by $y_i(t)$ for all $i\in I$ and $t\in\mathbb R_{>0}$. Let $T_{>0}$ be the identity component of $T(\mathbb R)$, and let $G_{\ge0}$ be the submonoid of $G$ generated by $U^\pm_{\ge 0}$ and $T_{>0}$.

We define
\[
\begin{aligned}
U^+_{>0}
&\coloneqq\{x_{i_1}(t_1)\cdots x_{i_l}(t_l) \mid t_1,\ldots,t_l>0\},\\
U^-_{>0}
&\coloneqq\{y_{i_1}(t_1)\cdots y_{i_l}(t_l) \mid t_1,\ldots,t_l>0\}
\end{aligned}
\]
for any reduced expression $w_0 = s_{i_1} \cdots s_{i_l}$, and we set
\[
G_{>0} \coloneqq U^-_{>0}T_{>0}U^+_{>0} = U^+_{>0}T_{>0}U^-_{>0}.
\]

\begin{example}
When $G = \SL_n(\CC)$ with the pinning as in \cref{eg:pinning}, the monoid $G_{\geq0}$ (respectively, $G_{>0}$) consists of all $n\times n$ matrices of determinant $1$ whose minors are all nonnegative (respectively, positive).
\end{example}

We set $\mathcal P^J_{>0} \coloneqq \{g\cdot P_J^+ \mid g\in U^-_{>0}\}$, and define $\mathcal P^J_{\geq0}$ to be the Hausdorff closure of $\mathcal P^J_{>0}$ in $\mathcal P^J(\mathbb R)$. We also set $\cQ^J_{>0} \coloneqq \cP^{\tau_I(J)}_{>0}$ (which equals $\{g\cdot P_J^- \mid g\in U^+_{>0}\}$) and $\cQ^J_{\ge 0} \coloneqq \cP^{\tau_I(J)}_{\geq0}$.

Following \cite{lusztig}, we similarly define the \newword{totally negative} and \newword{totally nonpositive parts} by replacing $\RR_{>0}$ with $\RR_{<0}$ above, i.e., by replacing the pinning of $G$ with the one obtained by inverting $x_i$ and $y_i$ for all $i\in I$. For example, $U^-_{<0} = (U^-_{>0})^{-1}$ and $\mathcal P_{<0}^J = \{g^{-1}\cdot P_J^+ \mid g\in U^-_{>0}\}$.

\begin{lemma}\label{negative_nonnegative_transverse}
Let $P\in\mathcal P^J_{<0}$ and $P'\in\cQ^{J'}_{\geq0}$. Then $P$ and $P'$ are transverse.
\end{lemma}

\begin{proof}
Since the maps $\cP_{<0}^\emptyset\to \cP^J_{<0}$ and $\cQ^{\emptyset}_{\geq 0} \to \cQ^{J'}_{\geq 0}$ are surjective, there exist Borel subgroups $B\subseteq P$ and $B'\subseteq P'$ such that $B\in\cP^\emptyset_{<0}$ and $B'\in\cP^\emptyset_{\ge 0}$. Then $B$ and $B'$ are opposed by \cite[Theorem 1.7(i)]{barkley_karp}, i.e., $B\cap B'$ is a maximal torus. Therefore $P$ and $P'$ are transverse by \cref{prop:transverseparabolics}\ref{prop:transverseparabolics_borels}.
\end{proof}

\begin{remark}\label{negate_pinning}
We note that the totally positive parts and totally negative parts of all spaces defined above are exchanged by inverting the pinning of $G$, i.e., by replacing $x_i(t)$ and $y_i(t)$ with $x_i(-t)$ and $y_i(-t)$, respectively. In particular, total negativity is the same notion as total positivity up to a simple change of conventions. (These two notions become interesting when they are considered simultaneously for the same choice of pinning.)
\end{remark}

\subsection{Flag amplituhedra and flagtopes}\label{sec:flagtopes_definition}
Inspired by the notions of amplituhedra and Grasstopes, we introduce the notion of flag amplituhedra and flagtopes, which serve as a Lie theoretic generalization of amplituhedra and Grasstopes.

\begin{definition}\label{flagtope_definition}
Fix $J,J'\subseteq I$, $P\in\mathcal P^J(\RR)$, and $\cS \subseteq \mathcal{O}_{J,J'}^P(\RR)$. Consider the image $\rho_P(\cS) \subseteq \cQ^{J\cap J'}(L)(\RR)$ under the $P$-restriction map \eqref{restriction_definition}. For certain naturally-arising $P$ and $\cS$ of interest we call $\rho_P(\cS)$ a \newword{flagtope}. In particular, in the case that $P\in\cP^J_{<0}$ and $\cS = \cQ^{J'}_{\ge 0}$ (so that $\cS \subseteq \mathcal{O}_{J,J'}^P$ by \cref{negative_nonnegative_transverse}), we call $\rho_P(\mathcal{Q}^{J'}_{\ge 0})$ a \newword{flag amplituhedron}.
\end{definition}

We note that \cref{flagtope_definition} proves the first part of \cref{theorem_intro}. We will verify the second part of \cref{theorem_intro} in \cref{sec:examples_amplituhedra}.

\begin{remark}\label{open_ended}
Our definition of flagtopes is intentionally open-ended. Our aim is that flagtopes should constitute a natural Lie-theoretic generalization of polytopes which also encompass the various known examples we discuss in \cref{sec:examples}. However, at this moment it is not clear what `natural' should mean. We hope our paper inspires future work on flagtopes which will clarify what further properties the definition should satisfy.
\end{remark}

\begin{remark}\label{flagtope_change_of_conventions}
We could have just as well defined a flag amplituhedron to be $\rho_P(\mathcal{Q}^{J'}_{\le 0})$ for some $P\in\mathcal{P}^J_{>0}$, which is equivalent to our original definition under the simple change of conventions from \cref{negate_pinning}. One reason for our choice of conventions in \cref{flagtope_definition} is that we find restricting $\mathcal{Q}^{J'}_{\ge 0}$ (rather than $\mathcal{Q}^{J'}_{\le0}$) to be simpler for performing calculations with the standard pinnings in the literature.
\end{remark}

\subsection{\texorpdfstring{$P$}{P}-restriction in terms of cosets}
Here we give an equivalent description of $P$-restriction that may be more amenable to computations. Recall that we have the isomorphism
\begin{align}\label{group_to_coset}
\mathcal Q^{J'}\xrightarrow[]{\sim}G/P_{J'}^-,\enspace g\cdot P_{J'}^- \mapsto gP_{J'}^-/P_{J'}^-.
\end{align}
Write $P = h\cdot P_J^+$ for some $h\in G$. Then under this isomorphism, $\cO^P_{J,J'}$ gets identified with a translated $P_J^+$-orbit in $G/P_{J'}^-$:
\[
\mathcal O_{J,J'}^P \simeq hP_J^+P_{J'}^-/P_{J'}^-.
\]
On the other hand, the codomain $\cQ^{J\cap J'}(L)$ can be identified with $(h\cdot L_J)/(h\cdot(L_J\cap P_{J'}^-))$.

Letting $\pi_J:P_J^+ = L_J\ltimes U_{P_J^+}\to L_J$ be the natural projection, the $P$-restriction map, under this identification, can be written as
\begin{align*}
\tilde{\rho}_P&: hP_J^+P_{J'}^-/P_{J'}^-\to (h\cdot L_J)/(h\cdot(L_J\cap P_{J'}^-)),\\
&\hphantom{{}:{}}hpP_{J'}^-/P_{J'}^-\mapsto (h\cdot(\pi_J(p)(L_J\cap P_{J'}^-)))/(h\cdot(L_J\cap P_{J'}^-)).
\end{align*}
We write $\tilde{\rho}$ to emphasize that this map uses cosets.

In the special case $h=\id$ we denote the corresponding map by $\tilde{\rho}_\standard$:
\[
\tilde{\rho}_\standard : P^+_JP_{J'}^-/P_{J'}^- \to L_J/(L_J\cap P_{J'}^-), \enspace pP_{J'}^-/P_{J'}^-\mapsto \pi_J(p)(L_J\cap P_{J'}^-)/(L_J\cap P_{J'}^-).
\]
We observe that we can phrase the flagtope construction purely in terms of $\tilde{\rho}_\standard$, by conjugating $P$ so that it becomes $P_J^+$:
\begin{proposition}\label{flagtope_standard}
Let $P = h\cdot P_J^+ \in\mathcal P^J$ and $\cS \subseteq \mathcal{O}_{J,J'}^P$. Let $\cS_\coset \subseteq G/P_{J'}^-$ denote the image of $\cS$ under the isomorphism from \eqref{group_to_coset}. Then
\[
\rho_P(\mathcal{S}) \simeq \rho_{P_J^+}(h^{-1}\cdot\mathcal{S}) \simeq \tilde{\rho}_\standard(h^{-1}\mathcal{S}_\coset).
\]
\end{proposition}

\begin{proof}
The first isomorphism follows by conjugating by $h^{-1}$. The second isomorphism follows from the identification above.
\end{proof}

In certain cases we can simplify the description in \cref{flagtope_standard} even further so that the image of the restriction map can be identified with a subset of an affine space. This is particularly useful for calculating examples, since it allows us to work in a single affine chart.

Given $J,J'\subseteq I$, define
\[
\tilde{\rho}_\affine : \prod_{\alpha\in\Phi^+\setminus\Phi_{J'}^+}U_\alpha = \prod_{\alpha \in \Phi^+_J \setminus\Phi^+_{J\cap J'}} U_\alpha \times \prod_{\alpha \in \Phi^+\setminus(\Phi^+_{J}\cup \Phi^+_{J'})} U_\alpha \to \prod_{\alpha \in \Phi^+_J\setminus \Phi^+_{J\cap J'}} U_\alpha
\]
to be projection onto the first factor. (We emphasize that while the map $\tilde{\rho}_\affine$ depends on both $J$ and $J'$, its domain only depends on $J'$.) For the domain of $\tilde{\rho}_\affine$, we have the isomorphism
\begin{align}\label{domain_affine}
\psi_{J'} : \prod_{\alpha\in\Phi^+\setminus\Phi_{J'}^+}U_\alpha \xrightarrow[]{\sim} B^+P^-_{J'}/P^-_{J'}, \enspace g\mapsto gP_{J'}^-/P_{J'}^-.
\end{align}
Similarly, the codomain of $\tilde{\rho}_\affine$ is isomorphic to
\[
(B^+\cap L_J)(L_J\cap P^-_{J'})/(L_J\cap P^-_{J'}) \subseteq L_J/(L_J\cap P^-_{J'}).
\]
Then $\tilde{\rho}_\affine$ is identified through these isomorphisms with the restriction of $\tilde{\rho}_\standard$ to this domain and codomain.
\begin{proposition}\label{flagtope_affine}
Let $P = h\cdot P_J^+ \in\mathcal P^J$ and $\cS \subseteq h\cdot (B^+\cdot P_{J'}^-)$. Let $\cS_\coset \subseteq G/P_{J'}^-$ denote the image of $\cS$ under the isomorphism from \eqref{group_to_coset}, and let $\psi_{J'}^{-1}(h^{-1}\mathcal{S}_\coset)$ denote the subset of the domain of $\tilde{\rho}_\affine$ corresponding to $h^{-1}\mathcal{S}_\coset \subseteq B^+P_{J'}^-/P_{J'}^-$ under the (inverse) isomorphism from \eqref{domain_affine}. Then $\mathcal{S}\subseteq\mathcal{O}_{J,J'}^P$, and
\[
\rho_P(\mathcal{S}) \simeq \tilde{\rho}_\affine(\psi_{J'}^{-1}(h^{-1}\mathcal{S}_\coset)).
\]
\end{proposition}

\begin{proof}
Since $P_J^+$ and $P_{J'}^-$ are transverse, and $B^+\subseteq P_J^+$, we have $(B^+\cdot P_{J'}^-) \subseteq \mathcal{O}_{J,J'}^{P_J^+}$. Conjugating by $h$ then implies that $\mathcal{S}\subseteq\mathcal{O}_{J,J'}^P$. We have $\rho_P(\mathcal{S}) \simeq \tilde{\rho}_\standard(h^{-1}\mathcal{S}_\coset)$ by \cref{flagtope_standard}, which is isomorphic to $\tilde{\rho}_\affine(\psi_{J'}^{-1}(h^{-1}\mathcal{S}_\coset))$ by the discussion above.
\end{proof}

\begin{remark}\label{flag_amplituhedra_affine}
We point out that \cref{flagtope_affine} can be applied to flag amplituhedra, i.e., when $P\in\cP^J_{<0}$ and $\cS = \cQ^{J'}_{\ge 0}$. Indeed, in this case we have $P = h\cdot P_J^+$ for some $h\in G_{<0}$; in order to apply \cref{flagtope_affine}, we need to check that $\cS \subseteq h\cdot (B^+\cdot P_{J'}^-)$. Equivalently, we need to show that every element of $\cS$ is transverse to $h\cdot B^+$. This follows from \cref{negative_nonnegative_transverse}, since $h\cdot B^+\in\cP^\emptyset_{<0}$. In fact, this argument applies to any $h\in U^-_{<0}B^+$ such that $P = h\cdot P_J^+$.
\end{remark}

\section{Explicit description of \texorpdfstring{$P$}{P}-restriction in type \texorpdfstring{$A$}{A}}\label{sec:type_A}

\noindent In this section we give a more explicit description of transversality and $P$-restriction when $G$ is of Dynkin type $A_{n-1}$. Accordingly, in this section we assume that $G = \SL_n(\CC)$ with the pinning from \cref{eg:pinning}.

\subsection{Partial flags and transversality in type \texorpdfstring{$A$}{A}}\label{sec:type_A_transversality}
We recall the well-known correspondence between parabolic subgroups of $\SL_n(\CC)$ and flags of subspaces, following \cite[Section 2.3]{barkley_karp}. Given $J \subseteq [n-1]$, write $J^c = \{k_1 < \cdots < k_l\}$, where $J^c$ denotes the complement of $J$ in $I = [n-1]$. Then we can identify $\mathcal{P}^J$ with the partial flag variety $\Fl_{k_1, \dots, k_l;n}(\CC)$ consisting of all tuples of subspaces
\[
F_\bullet = (0 \subset F_1 \subset \cdots \subset F_l \subset \CC^n),
\]
where each $F_i$ is a $k_i$-dimensional subspace of $\CC^n$. (Explicitly, $F_\bullet$ is identified with its stabilizer subgroup $\Stab(F_\bullet)\in\mathcal{P}^J$. Conversely, $g\cdot P_J^+\in\mathcal{P}^J$ corresponds to the flag $F_\bullet$, where each $F_i$ is the span of the first $k_i$ columns of $g$.) There are two special cases of particular interest:
\begin{itemize}[itemsep=2pt]
\item if $J=\emptyset$, then $\Fl_{1,\dots, n-1;n}(\CC)$ is the variety of \newword{complete flags} in $\CC^n$, which is isomorphic to the variety of Borel subgroups of $\SL_n(\CC)$;
\item if $J = [n-1]\setminus\{k\}$, then $\Fl_{k;n}(\CC)$ is the \newword{Grassmannian} of all $k$-dimensional subspaces of $\CC^n$, denoted $\Gr_{k,n}(\CC)$, which is isomorphic to the variety of maximal parabolic subgroups of $\SL_n(\CC)$ of type $J$.
\end{itemize}

We say that two partial flags $F_\bullet$ and $F'_\bullet$ are \newword{transverse} if and only if their stabilizer subgroups are transverse parabolic subgroups. We have the following characterization of transversality for subspaces:
\begin{lemma}[{\cite[Proposition 4.1]{barkley_karp}}]\label{transverse_subspaces}
Let $V\in\Gr_{k,n}$ and $V'\in\Gr_{k',n}$. Then $V$ and $V'$ are transverse if and only if
\[
\dim(V\cap V') = \max(0,k+k'-n), \enspace \text{ or equivalently, }\enspace \dim(V + V') = \min(n,k+k').
\]
That is, $V$ and $V'$ are transverse if and only if $V\cap V'$ is as small as possible, which is equivalent to $V+V'$ being as large as possible.
\end{lemma}

We show that transversality of subspaces similarly characterizes transversality for all partial flag varieties of $\SL_n(\CC)$:
\begin{proposition}\label{type_A_transversality}
Let $F_\bullet$ and $F'_\bullet$ be partial flags of $\CC^n$. Then $F_\bullet$ and $F'_\bullet$ are transverse if and only if $F_i$ and $F'_{i'}$ are transverse for all subspaces $F_i$ in $F_\bullet$ and $F'_{i'}$ in $F'_\bullet$.
\end{proposition}

\begin{proof}
As mentioned above, the result holds when $F_\bullet$ and $F'_\bullet$ each consist of a single subspace by \cite[Proposition 4.1]{barkley_karp}. The general result then follows from \cref{transversality_maximal_parabolics} in the case of $G = \SL_n(\CC)$. Alternatively, rather than applying \cref{transversality_maximal_parabolics}, we can use the well-known fact that the relative position of $F_\bullet$ and $F'_\bullet$ is uniquely determined by the dimensions $\dim(F_i \cap F'_{i'})$ for all $i$ and $i'$.
\end{proof}

\subsection{\texorpdfstring{$P$}{P}-restriction in type \texorpdfstring{$A$}{A}}\label{sec:type_A_restriction}
We now describe $P$-restriction when $G = \SL_n(\CC)$. We emphasize that our description is phrased solely in terms of types, without referring to cotypes. The reason is that in order to naturally associate a flag of subspaces to a parabolic subgroup of cotype $J$, we would want to use the map $\cdot^\perp$ from \eqref{perp_formula}. While $\cdot^\perp$ has a simple description on flags of subspaces (it is given by taking the orthogonal complement with respect to the standard bilinear pairing on $\CC^n$), we prefer to avoid using it for now.

Given $F_\bullet\in\Fl_{k_1, \dots, k_l; n}(\CC)$ with associated parabolic subgroup $P = \Stab(F_\bullet)$, let $L$ be the Levi factor of $P$. We can naturally identify parabolic subgroups $Q\subseteq L$ with partial flags of $\CC^n$ which refine $F_\bullet$; moreover, $\type_L(Q)$ is the complement in $I$ of the set of dimensions of the associated partial flag. Under these conventions, $P$-restriction is given as follows (as determined by Moonen and Wedhorn \cite[Example 3.3]{moonen_wedhorn04}):
\begin{proposition}\label{type_A_restriction}
Let $F_\bullet\in\Fl_{k_1, \dots, k_l;n}(\CC)$ and $F'_\bullet\in\Fl_{k_1', \dots, k'_{l'};n}(\CC)$ be transverse. Then $\rho_{F_\bullet}(F'_\bullet)$ is the partial flag obtained from the following sequence of nested subspaces by deleting any repetitions:
\[
\begin{array}{rcccccccl}
0 &\subseteq& F_1\cap F'_1 &\subseteq& \cdots &\subseteq& F_1\cap F'_{l'} &\subseteq& \\
F_1 &\subseteq& F_1+(F_2\cap F'_1) &\subseteq& \cdots &\subseteq& F_1+(F_2\cap F'_{l'}) &\subseteq& \\
F_2 &\subseteq& F_2+(F_3\cap F'_1) &\subseteq& \cdots &\subseteq& F_2+(F_3\cap F'_{l'}) &\subseteq& \\
&&&& \vdots \\
F_l &\subseteq& F_l+F'_1 &\subseteq& \cdots &\subseteq& F_l+F'_{l'} &\subseteq& \CC^n.
\end{array}
\]
This flag is obtained from $F_\bullet$ by including the subspace
\[
F_i + (F_{i+1}\cap F'_{i'}) \enspace \text{(which has dimension $k_i + k_{i+1} + k'_{i'} - n$)}
\]
for all $i\in \{0,\ldots,l\}$ and $i'\in [l']$ such that $k_i < n - k'_{i'} < k_{i+1}$. (By convention, $F_0 = 0$, $F_{l+1} = \CC^n$, $k_0 = 0$, $k_{l+1} = n$.)
\end{proposition}

\begin{proof}
The first part of the statement is from \cite[Example 3.3]{moonen_wedhorn04}. For the second part of the statement, let $J = \{k_1, \dots, k_l\}^c = \type_G(F_\bullet)$ and $\tilde{J} = \{k_1', \dots, k_{l'}'\}^c = \type_G(F'_\bullet)$. Then the set of dimensions of the subspaces appearing in the flag $\rho_{F_\bullet}(F'_\bullet)$ is $\type_L(\rho_{F_\bullet}(F'_\bullet))^c$, which by \eqref{type_formula} equals
\[
J^c \cup \tau_J(J\cap\tau_I(\tilde{J}^c)) = \{k_1, \dots, k_l\} \cup \tau_J(J \cap \{n-k_{l'}', \dots, n-k_1'\}).
\]
Note that $\tau_J(p) = k_i + k_{i+1} - p$ for all $p\in (k_i,k_{i+1})$. Accordingly, given $i'\in[l']$ such that $n-k'_{i'}\in J$, take $i\in \{0,\ldots,l\}$ such that $n-k'_{i'}\in (k_i,k_{i+1})$. Then $\tau_J(n-k'_{i'}) = k_i + k_{i+1} + k'_{i'} - n$, so there exists a subspace in $\rho_{F_\bullet}(F'_\bullet)$ of this dimension. We can check using \cref{transverse_subspaces} that the corresponding subspace is $F_i + (F_{i+1}\cap F'_{i'})$.
\end{proof}

\begin{remark}
In \cref{type_A_restriction}, one should technically interpret the image partial flag as a sequence of partial flags in the ambient vector spaces $F_1, F_2/F_1, \dots, F_l/F_{l-1}, \mathbb C^n/F_l$. We work with the single partial flag in $\CC^n$ for simplicity.
\end{remark}

\begin{example}\label{eg:type_A_restriction}
Let $V\in\Gr_{k,n}(\CC)$ and $V'\in\Gr_{k',n}(\CC)$ be transverse, where $1 \le k,k' \le n-1$. Then $\rho_V(V')$ depends on how $k+k'$ compares with $n$:
\begin{itemize}[itemsep=2pt]
\item If $k+k' < n$, then $\rho_V(V')$ is the partial flag
\[
0 \subset V \subset V + V' \subset \CC^n,
\]
which we can regard as an element of $\Gr_{k'}(\CC^n/V)$.
\item If $k+k' = n$, then $\rho_V(V') = V$, which we can regard as a point.
\item If $k+k' > n$, then $\rho_V(V')$ is the partial flag
\[
0 \subset V \cap V' \subset V \subset \CC^n,
\]
which we can regard as an element of $\Gr_{k+k'-n}(V)$.\qedhere
\end{itemize}
\end{example}

\subsection{Total positivity vs.\ Pl\"{u}cker positivity in type \texorpdfstring{$A$}{A}}\label{type_A_positivity}
For partial flag varieties of $G = \SL_n(\CC)$, we have two interesting notions of positivity. One is Lusztig's total positivity, defined in \cref{sec:total_positivity}. To define the other, we recall the Pl\"{u}cker embedding of $\Gr_{k,n}(\CC)$. Let $\binom{[n]}{k}$ denote the set of $k$-element subsets of $[n]$, and given $V\in\Gr_{k,n}(\CC)$, take an $n\times k$ matrix $A$ whose columns span $V$. Then the \newword{Pl\"{u}cker coordinate} $\Delta_K(V)$ (for $K\in\binom{[n]}{k}$) is defined to be the $k\times k$ minor of $A$ located in the rows $K$; these coordinates are well-defined up to global rescaling.

We say that a flag $F_\bullet$ is \newword{Pl\"{u}cker-nonnegative} (respectively, \newword{Pl\"{u}cker-positive}) if and only if each $F_i$ has nonnegative (respectively, positive) Pl\"{u}cker coordinates up to rescaling. We then define $F_\bullet$ to be \newword{Pl\"{u}cker-nonpositive} (respectively, \newword{Pl\"{u}cker-negative}) if and only if $\delta(F_\bullet)$ is Pl\"{u}cker-nonnegative (respectively, Pl\"{u}cker-positive), where $\delta$ denotes the linear map $(v_1, \dots, v_n) \mapsto (v_1, -v_2, v_3, \dots, (-1)^{n-1}v_n)$ on $\CC^n$. (Note that $\delta$ similarly takes totally nonpositive flags to totally nonnegative flags, and takes totally negative flags to totally positive flags.)

We denote total positivity by `$>\!0$' and Pl\"{u}cker positivity by `$\Delta\!>\!0$'. We have
\[
\Fl_{k_1, \dots, k_l;n}^{\ge 0} \subseteq \Fl_{k_1, \dots, k_l;n}^{\Delta\ge 0} \enspace\text{ and }\enspace \Fl_{k_1, \dots, k_l;n}^{>0} \subseteq \Fl_{k_1, \dots, k_l;n}^{\Delta > 0};
\]
by \cite[Theorem 1.1]{bloch_karp23} the containments are equalities if the integers $k_1, \dots, k_l$ are consecutive, and otherwise the containments are strict. (See \cite{barkley_boretsky_eur_gao26} for results in other Lie types.) We have the following analogue of \cref{negative_nonnegative_transverse}:
\begin{lemma}\label{plucker_negative_nonnegative_transverse}
Let $F_\bullet\in\Fl_{k_1, \dots, k_l;n}^{\Delta<0}$ and $F'_\bullet\in\Fl_{k_1', \dots, k_{l'}';n}^{\Delta \ge 0}$. Then $F_\bullet$ and $F'_\bullet$ are transverse.
\end{lemma}

\begin{proof}
By \cref{type_A_transversality}, the statement reduces to the case that $F_\bullet$ and $F'_\bullet$ each consist of a single subspace, i.e., both partial flag varieties are Grassmannians. This case of the statement follows from \cref{negative_nonnegative_transverse}, since the notions of total positivity and Pl\"{u}cker positivity coincide for Grassmannians.
\end{proof}

Accordingly, given $F_\bullet\in\Fl_{k_1, \dots, k_l;n}^{\Delta<0}$ and $1 \leq k'_1 < \cdots < k'_{l'} \le n-1$, we define the \newword{Pl\"{u}cker-nonnegative flag amplituhedron} to be the image $\rho_{F_\bullet}(\Fl_{k_1', \dots, k_{l'}';n}^{\Delta \ge 0})$, which \cref{plucker_negative_nonnegative_transverse} shows is well-defined. In this context, we call flag amplituhedra as defined in \cref{flagtope_definition} \newword{totally nonnegative flag amplituhedra}, in order to distinguish them from Pl\"{u}cker-nonnegative flag amplituhedra. Pl\"{u}cker-nonnegative flag amplituhedra arise naturally in the physics of scattering amplitudes; see \cref{sec:examples_loop_amplituhedra}.

\section{Examples of flagtopes in the literature}\label{sec:examples}

\noindent In this section we explain how various spaces studied recently in scattering amplitudes, positive geometries, or combinatorics are instances of flagtopes. This shows that the flagtope setup is quite general and can provide a unifying framework in which to study these spaces.

\subsection{(Tree) amplituhedra and Grassmann polytopes}\label{sec:examples_amplituhedra}
Let $k,m,n\ge 0$ with $k+m \le n$, and let $Z$ be a $(k+m)\times n$ matrix whose $(k+m)\times (k+m)$ minors are all positive. We view $Z$ as a linear map $\mathbb{C}^n \to \mathbb{C}^{k+m}$, which induces a rational map on Grassmannians $Z_{\Gr} : \Gr_{k,n}(\CC) \dashrightarrow \Gr_{k,k+m}(\CC)$. The map $Z_{\Gr}$ is well-defined on the totally nonnegative part $\Gr_{k,n}^{\ge 0}$ (this is a special case of \cref{plucker_negative_nonnegative_transverse}); we define the \newword{(tree) amplituhedron}
\[
\mathcal{A}_{n,k,m}(Z) \coloneqq Z_{\Gr}(\Gr_{k,n}^{\ge 0}) \subseteq \Gr_{k,k+m}(\CC).
\]

The amplituhedron was introduced by Arkani-Hamed and Trnka \cite[Section 9]{arkani-hamed_trnka14} as part of a novel geometric approach to the study of scattering amplitudes, building on work of Arkani-Hamed, Bourjaily, Cachazo, Goncharov, Postnikov, and Trnka \cite{arkani-hamed_bourjaily_cachazo_goncharov_postnikov_trnka16}. Specifically, $\mathcal{A}_{n,k,m}(Z)$ for $m=4$ is conjecturally a \emph{positive geometry} whose canonical differential form is the tree-level scattering amplitude in planar $\mathcal{N} = 4$ supersymmetric Yang--Mills theory.

Following \cite[Section 3.2]{karp_williams19}, we can reformulate the amplituhedron construction as below. We will need an elementary fact about orthogonal complements:
\begin{lemma}[{\cite[(3.2)]{barkley_karp}}]\label{orthogonal_complement}
Let $V\in\Gr_{k,n}(\CC)$, and let $V^\perp\in\Gr_{n-k,n}(\CC)$ be the orthogonal complement of $V$ with respect to the standard bilinear pairing on $\CC^n$.
\begin{enumerate}[label=(\roman*), leftmargin=*, itemsep=2pt]
\item\label{orthogonal_complement_nonnegative} We have $V\in\Gr_{k,n}^{\ge 0}$ if and only if $V^\perp\in\Gr_{n-k,n}^{\le 0}$.
\item\label{orthogonal_complement_positive} We have $V\in\Gr_{k,n}^{>0}$ if and only if $V^\perp\in\Gr_{n-k,n}^{<0}$.\hfill\qed
\end{enumerate}
\end{lemma}

Let $W\in\Gr_{k+m,n}(\CC)$ denote the row span of $Z$; then the $(k+m)\times (k+m)$ minors of $Z$ are all positive (or all negative) if and only if $W\in\Gr_{k+m,n}^{>0}$. Since $\ker(Z) = W^\perp$, we calculate that
\[
\mathcal{A}_{n,k,m}(Z) \simeq \{V + W^\perp \mid V\in\Gr_{k,n}^{\ge 0}\},
\]
and the right-hand side is precisely the flag amplituhedron $\rho_{W^\perp}(\Gr_{k,n}^{\ge 0})$ under the conventions of \cref{type_A_restriction}. (Note that $W^\perp\in\Gr_{n-k-m,n}^{<0}$ by \cref{orthogonal_complement}\ref{orthogonal_complement_positive}.)

Alternatively, since $(V+W^\perp)^\perp = V^\perp\cap W$, we have
\[
\mathcal{A}_{n,k,m}(Z) \simeq \{V^\perp\cap W \mid V\in\Gr_{k,n}^{\ge 0}\} = \{V' \cap W \mid V' \in\Gr_{n-k,n}^{\le 0}\},
\]
where the last equality follows from \cref{orthogonal_complement}\ref{orthogonal_complement_nonnegative}. The right-hand side above (which was denoted $\mathcal{B}_{n,k,m}(W)$ in \cite[Section 3.2]{karp_williams19}) is similarly precisely the flag amplituhedron $\rho_W(\Gr_{n-k,n}^{\le 0})$ (see \cref{flagtope_change_of_conventions}). We point out that this discussion (along with \cref{eg:type_A_restriction}) proves the second part of \cref{theorem_intro}, i.e., that every flag amplituhedron associated to a pair of type $A$ Grassmannians is in fact an amplituhedron defined by Arkani-Hamed and Trnka.

Lam \cite{lam16} introduced \newword{Grassmann polytopes} (also called \newword{Grasstopes} \cite{mandelshtam_pavlov_pratt25}), which are obtained by replacing $W^\perp$ above with an arbitrary element of $\Gr_{n-k-m,n}(\mathbb{R})$ which is transverse to every element of $\Gr_{k,n}^{\ge 0}$, and replacing $\Gr_{k,n}^{\ge 0}$ with an arbitrary cell in its cell decomposition. Mandelshtam, Pavlov, and Pratt then introduced \newword{rational Grassmann polytopes}, which are obtained by replacing $W^\perp$ with an arbitrary element of $\Gr_{n-k-m,n}(\mathbb{R})$ (with no transversality condition) and restricting the domain to elements of $\Gr_{k,n}^{\ge 0}$ (or a cell thereof) which are transverse to $W^\perp$. All of these constructions are therefore examples of flagtopes.

\subsection{One-loop amplituhedra}\label{sec:examples_loop_amplituhedra}
As we mentioned above, the (tree) amplituhedron encodes the tree-level term of a certain scattering amplitude. This scattering amplitude has an expansion coming from Feynman diagrams whose terms are indexed by the \emph{loop order} $L \ge 0$, where the leading term (also called the tree-level term) is indexed by $L=0$. Arkani-Hamed and Trnka \cite[Section 9]{arkani-hamed_trnka14} similarly defined an \newword{$L$-loop amplituhedron} $\mathcal{A}_{n,k,m;L}(Z)$ which for $m=4$ encodes (conjecturally) the integrand of the term at loop order $L$, where $\mathcal{A}_{n,k,m;0}(Z) = \mathcal{A}_{n,k,m}(Z)$.

The one-loop amplituhedron is defined (for $m\ge 2$) to be
\[
\mathcal{A}_{n,k,m;1}(Z) \coloneqq \{0 \subset Z_{\Gr}(F_1) \subset Z_{\Gr}(F_2) \subset \CC^{k+m} \mid F_\bullet \in\Fl_{k,k+2;n}^{\Delta\ge 0}\}\subseteq \Fl_{k,k+2;k+m}(\CC).
\]
Just as in \cref{sec:examples_amplituhedra}, we can show that
\begin{align}\label{one_loop_amplituhedron}
\mathcal{A}_{n,k,m;1}(Z) \simeq \rho_{W^\perp}(\Fl_{k,k+2;n}^{\Delta\ge 0}),
\end{align}
which is an instance of a Pl\"{u}cker-nonnegative flag amplituhedron.

This is one motivation for studying Pl\"{u}cker-nonnegative flag amplituhedra as opposed to totally nonnegative flag amplituhedra. We emphasize that it would be unsuitable for the connection with physics to replace $\Fl_{k,k+2;n}^{\Delta\ge 0}$ with $\Fl_{k,k+2;n}^{\ge 0}$. Indeed, $\Fl_{k,k+2;n}^{\Delta\ge 0}$ admits a cyclic action of order $n$ induced by the linear map $(v_1, \dots, v_n) \mapsto (v_2, \dots, v_n, (-1)^{k-1}v_1)$ on $\CC^n$, which reflects a cyclic symmetry among the $n$ particles defining the scattering amplitude. On the other hand, this cyclic action does not preserve $\Fl_{k,k+2;n}^{>0}$ (see \cite[Theorem 1.3]{bloch_karp23}).

While the $L$-loop amplituhedron $\mathcal{A}_{n,k,m;L}(Z)$ for $L\ge 2$ (whose definition we do not go into here) is not naturally a flagtope as defined in \cref{flagtope_definition}, one could generalize the definition to construct `$L$-loop flagtopes'. This is beyond the scope of our paper.

\subsection{ABJM amplituhedra}\label{sec:examples_orthogonal}
The \newword{ABJM amplituhedron} was independently introduced by Huang, Kojima, Wen, and Zhang \cite{huang_kojima_wen_zhang22} and by He, Kuo, and Zhang \cite{he_kuo_zhang22}, as follows. Huang and Wen \cite{huang_wen14} (see also \cite{huang_wen_xie14,galashin_pylyavskyy20}) defined the \newword{totally nonnegative orthogonal Grassmannian} $\OG_{n,2n}^{\ge 0}$ to be the subset of $\Gr_{n,2n}^{\ge 0}$ cut out by the equations $\Delta_K = \Delta_{[2n]\setminus K}$ for all $K\in\binom{[2n]}{n}$. (We emphasize that this is \emph{not} the totally nonnegative part one would obtain by taking $G = \SO_{2n}(\CC)$ in \cref{sec:total_positivity}; see \cite[Remark 5.2]{galashin_pylyavskyy20} and \cite[Section 6.3]{boretsky_calvo_cortes_el_maazouz}.) Then the ABJM amplituhedron is defined to be $Z_{\text{Gr}}(\OG_{n,2n}^{\ge 0}) \subseteq \Gr_{n,n+2}(\CC)$, where $Z$ is a $(n+2)\times(2n)$ real matrix whose $(n+2)\times (n+2)$ minors are all positive. As in \cref{sec:examples_amplituhedra}, we have
\[
Z_{\text{Gr}}(\OG_{n,2n}^{\ge 0}) \simeq \rho_{\ker(Z)}(\OG_{n,2n}^{\ge 0}),
\]
so the ABJM amplituhedron is a special case of a flagtope. The motivation for studying ABJM amplituhedra is that they encode scattering amplitudes in ABJM theory \cite{aharony_bergman_jafferis_maldacena08}.

\subsection{Lagrangian amplituhedra}\label{sec_examples_lagrangian}
Krovi and Tessler \cite{krovi_tessler} recently introduced the \newword{Lagrangian amplituhedron}, following the definition of the ABJM amplituhedron from \cref{sec:examples_orthogonal}. Let $\LG_{n,2n}(\CC)$ denote the Lagrangian Grassmannian for the group $G = \Sp_{2n}(\CC)$. Karpman \cite[Theorem 5.16 and Section 6]{karpman18} showed we can identify $\LG_{n,2n}^{\ge 0}$ with the subset of $\Gr_{n,2n}^{\ge 0}$ cut out by the equations $\Delta_{K} = \Delta_{R(K)}$ for all $K\in\binom{[2n]}{n}$, where $R(K) \coloneqq [2n]\setminus\{2n+1-i \mid i\in K\}$. Then the Lagrangian amplituhedron is defined to be $Z_{\text{Gr}}(\LG_{n,2n}^{\ge 0}) \subseteq \Gr_{n,n+2}(\CC)$, where $Z$ is a $(n+2)\times(2n)$ real matrix whose $(n+2)\times (n+2)$ minors are all positive. As in \cref{sec:examples_amplituhedra}, we have
\begin{align}\label{lagrangian_amplituhedron_definition}
Z_{\text{Gr}}(\LG_{n,2n}^{\ge 0}) \simeq \rho_{\ker(Z)}(\LG_{n,2n}^{\ge 0}),
\end{align}
so the Lagrangian amplituhedron is a special case of a flagtope.

We emphasize that in \eqref{lagrangian_amplituhedron_definition} we are working with $G = \SL_{2n}(\CC)$. In particular, the definition of the Lagrangian amplituhedron is different than our definition (in \cref{flagtope_definition}) of the flag amplituhedron $\rho_P(\LG_{n,2n}^{\ge 0})$ in the case $G = \Sp_{2n}(\CC)$.

\section{Open problems}\label{sec:problems}

\noindent In this section we propose some open problems about flagtopes.

\subsection{Positive geometries}\label{sec:problems_positive_geometries}
Inspired by the amplituhedron construction, Arkani-Hamed, Bai, and Lam \cite{arkani-hamed_bai_lam17} introduced the framework of \newword{positive geometries}. Informally, a positive geometry is a positive subset of a complex algebraic variety which has associated to it a \newword{canonical differential form}, with simple poles along the boundary components of the positive geometry. The amplituhedron $\mathcal{A}_{n,k,m}(Z)$ is conjecturally a positive geometry (see \cite{koefler_pavlov_sinn} for recent progress) whose canonical differential form when $m=4$ is the tree-level scattering amplitude in planar $\mathcal{N}=4$ supersymmetric Yang--Mills theory. This framework has since been applied to many other scattering amplitudes and wavefunctions in physics \cite{arkani-hamed_bai_he_yan18,arkani-hamed_benincasa_postnikov17}, and found connections with matroids \cite{eur_lam26} and Hodge theory \cite{brown_dupont25}; we refer to Lam's survey \cite{lam26} for further details.

We expect that flag amplituhedra are also positive geometries:
\begin{problem}\label{problem_positive_geometry}
Show that every flag amplituhedron is a positive geometry, and calculate its canonical form.
\end{problem}

In \cref{sec:hexagon_form} we show that the flag amplituhedron from \cref{figure_hexagon} is a positive geometry and determine its canonical form. We also propose studying more general flagtopes from the perspective of positive geometry.

\subsection{Tilings}\label{sec:problems_tilings}
Another important open problem about amplituhedra is to find their `tilings'. The totally nonnegative complete flag variety $\cP^\emptyset_{\ge 0}$ has a cell decomposition \cite{rietsch99}
\[
\cP^\emptyset_{\ge 0} = \bigsqcup_{v \le w}R_{v,w}^{>0},
\]
where the union is over all $v\le w$ in $W$, and $R_{v,w}^{>0}$ is the totally positive part of the open Richardson variety $(B^-\dot{v}\cdot B^+) \cap (B^+\dot{w}\cdot B^+)$ introduced in \cite{kazhdan_lusztig79}. Similarly, we have the cell decomposition of an arbitrary totally nonnegative partial flag variety
\begin{align}\label{cell_decomposition}
\cP^J_{\ge 0} = \bigsqcup_{\substack{v \le w, \\ w\in W^J_{\text{min}}}}\Pi_{v,w}^{>0},
\end{align}
where the union is over all $v\le w$ in $W$ such that $w$ is a minimal-length coset representative of $W/W_J$, and $\Pi_{v,w}^{>0}$ is the projection of $R_{v,w}^{>0}$ to $\cP^J$.

Consider a flag amplituhedron $\rho_P(\cQ^{J'}_{\ge 0})$, and let $d$ denote the dimension of the image partial flag variety $\cQ^{J\cap J'}(L)$. A \newword{tile} of $\rho_P(\cQ^{J'}_{\ge 0})$ is the image of a $d$-dimensional cell of $\cQ^{J'}_{\ge 0}$ on which $\rho_P$ is injective, and a \newword{tiling} of $\rho_P(\cQ^{J'}_{\ge 0})$ is a collection of tiles which are pairwise disjoint and whose union is dense in $\rho_P(\cQ^{J'}_{\ge 0})$ (in the Hausdorff topology).

Tilings of $\mathcal{A}_{n,k,4}(Z)$ are of particular interest in physics because they lead to different expressions for the associated scattering amplitude. In fact, the original definition in \cite{arkani-hamed_trnka14} was predicated on a conjecture that the BCFW recursion of \cite{britto_cachazo_feng05,britto_cachazo_feng_witten05} can be used to construct tilings of $\mathcal{A}_{n,k,4}(Z)$, which was recently proved in \cite{even-zohar_lakrec_parisi_tessler_sherman-bennett_williams,galashin,galashina}. (See \cite{karp_williams_zhang20,even-zohar_lakrec_tessler25,tessler} for earlier and related work, and \cite{oren-perlstein_tessler} for tilings of ABJM amplituhedra.) We pose the problem of finding tilings of flag amplituhedra:
\begin{problem}\label{problem_tilings}
Find tilings of flag amplituhedra.
\end{problem}

In \cref{sec:hexagon_tilings} we construct two tilings for the flag amplituhedron from \cref{figure_hexagon}. More generally, we pose the problem of finding cell decompositions of flag amplituhedra (which requires identifying the lower-dimensional `tiles'), and studying these problems for general flagtopes.

\subsection{Topology}\label{sec:problems_topology}
Topologically, totally nonnegative flag varieties (and their cell closures) are homeomorphic to closed Euclidean balls \cite{galashin_karp_lam19,galashin_karp_lam22,bao_he24}, and it is expected that every amplituhedron $\mathcal{A}
_{n,k,m}(Z)$ is homeomorphic to a closed ball. (This is only known in certain special cases \cite{sturmfels88,karp_williams19,blagojevic_galashin_palic_ziegler19,bloch_karp23a}.) We pose the analogous problem for flag amplituhedra:
\begin{problem}\label{problem_bal}
Is every flag amplituhedron homeomorphic to a closed ball?
\end{problem}

We see from \cref{figure_hexagon} that the flag amplituhedron depicted there is homeomorphic to a closed ball. We also propose studying the topology of more general flagtopes. See \cite{mandelshtam_pavlov_pratt25} for examples of Grassmann polytopes with nontrivial topology.

\section{Example of a flag amplituhedron: a curved hexagon}\label{sec:hexagon}

\noindent We calculate in detail an example of a flag amplituhedron which is not a Grassmannian amplituhedron in the sense of \cref{sec:examples_amplituhedra}.
\subsection{Setup}
Following the setup of \cref{type_A_restriction}, we take $G = \SL_4(\CC)$, fix $V\in\Gr_{2,4}^{<0}$, and consider the flag amplituhedron obtained by restricting elements $F_\bullet'\in\Fl_{1,3;4}(\CC)$. (In the notation of \cref{sec:flagtopes}, we have $I = \{1,2,3\}$ and
\begin{align*}
J &= \type_G(\Stab(V)) = \{2\}^c = \{1,3\}, \\
\tilde{J} &= \type_G(\Stab(F_\bullet')) = \{1,3\}^c = \{2\}, \\
J' &= \cotype_G(\Stab(F_\bullet')) = \tau_I(\tilde{J}) = \tau_I(\{2\}) = \{2\}.\text{)}
\end{align*}

Since total positivity is strictly stronger than Pl\"{u}cker positivity for $\Fl_{1,3;4}(\CC)$, we have two (potentially) different flag amplituhedra:
\begin{itemize}[itemsep=2pt]
\item the totally nonnegative flag amplituhedron $\rho_V(\Fl_{1,3;4}^{\ge 0})$ (which we denote by $\mathcal{A}_{\ge 0}$); and
\item the Pl\"{u}cker-nonnegative flag amplituhedron $\rho_V(\Fl_{1,3;4}^{\Delta\ge 0})$ (which we denote by $\mathcal{A}_{\Delta\ge 0}$).
\end{itemize}
(Since $\Gr_{2,4}^{\Delta<0} = \Gr_{2,4}^{<0}$, we do not need to make the same distinction for $V$.) In fact, we will show later that $\mathcal{A}_{\ge 0} = \mathcal{A}_{\Delta\ge 0}$. However, we point out that in general, $\mathcal{A}_{\ge 0}$ can be strictly contained in $\mathcal{A}_{\Delta\ge 0}$. For example, if we replace $J = \{1,3\}$ with $J = \{1,2,3\}$ (as in \cref{restriction_examples}\ref{restriction_examples_identity}), then the restriction map becomes the identity map, so $\mathcal{A}_{\ge 0} = \Fl_{1,3;4}^{\ge 0}$ and $\mathcal{A}_{\Delta\ge 0} = \Fl_{1,3;4}^{\Delta\ge 0}$.

We take $V$ to be given as follows:
\[
V\in\Gr_{2,4}^{<0} \text{ is represented by } \begin{bmatrix}
1 & 0 \\
-1 & -1 \\
1 & 2 \\
-1 & -3
\end{bmatrix},
\]
i.e., $V$ is the column span of the matrix above. The fact that $V\in\Gr_{2,4}^{<0}$ is equivalent to $\delta(V)\in\Gr_{2,4}^{>0}$, i.e., $(-1)^{i+j}\Delta_{i,j}(V) > 0$ for all $1 \le i < j \le 4$. Since
\[
\Delta_{1,2} = -1, \enspace
\Delta_{1,3} = 2, \enspace
\Delta_{1,4} = -3, \enspace
\Delta_{2,3} = -1, \enspace
\Delta_{2,4} = 2, \enspace
\Delta_{3,4} = -1,
\]
we see that indeed $V\in\Gr_{2,4}^{<0}$.

In the notation of \cref{type_A_restriction}, we have $n=4$, $k_1 = 2$, and $(k'_1, k'_2) = (1,3)$. Hence $\rho_V(F_\bullet')$ is the partial flag obtained from $V$ by including the following subspaces:
\begin{itemize}[itemsep=4pt]
\item $V\cap F'_2$, since $n-k'_2 = 1$ is contained in the interval $(0,2) = (k_0,k_1)$; and
\item $V + F'_1$, since $n-k'_1 = 3$ is contained in the interval $(2,4) = (k_1,k_2)$.
\end{itemize}
The codomain of $\rho_V(F_\bullet')$ is isomorphic to $\mathbb{P}(V) \times \mathbb{P}(\CC^4/V) \cong \mathbb{P}^1\times\mathbb{P}^1$. (It may appear from \eqref{one_loop_amplituhedron} that $\mathcal{A}_{\Delta\ge 0}$ is an instance of a one-loop amplituhedron; it turns out it is not, because the definition of $\mathcal{A}_{n,k,m;1}(Z)$ requires that $m\ge 2$, and we would need to take $m=1$ to get $\mathcal{A}_{\Delta\ge 0}$. This is reflected in the fact that $\rho_V(F_\bullet')$ is computed by taking one intersection and one sum, rather than two sums.)

In order to calculate $\rho_V(F_\bullet')$, we write a generic element $F'_\bullet$ as follows:
\begin{align}\label{hexagon_domain_element}
F'_\bullet\in\Fl_{1,3;4}(\CC) \text{ is represented by } \begin{bmatrix}
1 & 0 & 0 \\
a & 1 & 0 \\
b & 0 & 1 \\
c & -d & e
\end{bmatrix},
\end{align}
i.e., $F'_1$ is the span of the first column of the matrix above, and $F'_2$ is the span of all three columns.

We find that $V\cap F'_2$ is spanned by
\[
(d+2e+3)\begin{bmatrix}
1 \\
-1 \\
1 \\
-1
\end{bmatrix} - (ad-be+c+d+e+1)\begin{bmatrix}
0 \\
-1 \\
2 \\
-3
\end{bmatrix}.
\]
We can choose coordinates on $\PP(V)\simeq\PP^1$ so that this is the point
\[
\frac{d+2e+3}{ad-be+c+d+e+1} \in \PP^1 \simeq \PP(V).
\]
Similarly, the element of $\PP(\CC^4/V)$ is spanned by $\transpose{(1,a,b,c)}$, or equivalently by
\[
(3a+3b+c+1)\begin{bmatrix}
-2 \\
-1 \\
4 \\
3
\end{bmatrix} + (6a+3b+3)\begin{bmatrix}
5/3 \\
5/3 \\
-5/3 \\
-5/3
\end{bmatrix},
\]
which we can identify with the point
\[
\frac{6a+3b+3}{3a+3b+c+1}\in\PP^1 \simeq \PP(\CC^4/V).
\]
Therefore (up to a change of coordinates on the codomain), the $V$-restriction map is given by
\begin{align}\label{hexagon_restriction_formula}
\rho_V(F'_\bullet) = \left(\frac{d+2e+3}{ad-be+c+d+e+1} \,,\, \frac{6a+3b+3}{3a+3b+c+1}\right) \in \PP^1 \times \PP^1
\end{align}
for $F'_\bullet$ as in \eqref{hexagon_domain_element}. 
\begin{remark}\label{affineexample}
    We chose our coordinates so that the flag amplituhedron will be contained in $\RR \times \RR \subseteq \PP^1 \times \PP^1$ and so that it exhibits reflective symmetries. One can alternatively run this computation using the formalism from \Cref{flagtope_affine} and \cref{flag_amplituhedra_affine} so that the flag amplituhedron is manifestly contained in an affine cell. In that setting, we can use 
\[ \cS = \begin{bmatrix}
1 & 0 & 0 & 0\\
a & 1 & 0  & 0 \\
b & 0 & 1 & 0\\
c & -d & e & 1
\end{bmatrix} \cdot  P_{\tilde{J}}^+
\enspace \text{ and } \enspace
h = \begin{bmatrix}
    1 & 0 & 0 & 0 \\
    -1 & 1 & 0 & 0 \\
    1 & -2 & 1 & 0 \\
    -1 & 3 & -3 & 1
\end{bmatrix} \in U^-_{<0},
\]
so that $V$ corresponds to $h\cdot P^+_{J}$. We have
\[ \widetilde \rho_\affine (\psi_{J'}^{-1}(h^{-1}\cS_\coset)) =  \begin{bmatrix}
    1 & \frac{d+2e+3}{ad-be+c+d+e+1} & 0 & 0 \\ 
    0 & 1 & 0 & 0 \\
    0 & 0 & 1 & \frac{2a+b+1}{3a+3b+c+1} \\
    0 & 0 & 0 &1
\end{bmatrix},  \]
so we obtain the same coordinates on the codomain (up to rescaling the $y$-coordinate by $3$). Note that if we were to choose a different $h$ with $V = h\cdot P_J^+$, we would potentially get a different parametrization, but the resulting coordinates would differ by fractional linear transformations.
\end{remark}

Since the nonnegative part of $\Fl_{1,3;4}(\CC)$ is the closure of its positive part (both for total positivity and for Pl\"{u}cker positivity), it suffices to calculate the image of the positive part and then take the closure. We calculate that the Pl\"{u}cker coordinates of $F'_\bullet$ are
\begin{align}\label{hexagon_pluckers}
\begin{gathered}
\Delta_1 = 1, \enspace
\Delta_2 = a, \enspace
\Delta_3 = b, \enspace
\Delta_4 = c, \\
\Delta_{1,2,3} = 1, \enspace
\Delta_{1,2,4} = e, \enspace
\Delta_{1,3,4} = d, \enspace
\Delta_{2,3,4} = ad-be+c.
\end{gathered}
\end{align}
In particular, $F_\bullet'\in\Fl_{1,3;4}^{\Delta > 0}$ if and only if
\begin{align}\label{plucker_positive_134}
a,b,c,d,e,ad-be+c>0.
\end{align}
By \cite[Example 10.1]{chevalier11}, we have $F'_\bullet\in\Fl_{1,3;4}^{>0}$ if and only if \eqref{plucker_positive_134} holds and additionally $be-c>0$.

We claim that both $\mathcal{A}_{\ge 0}$ and $\mathcal{A}_{\Delta\ge 0}$ are equal to the curved hexagon shown in \cref{figure_hexagon}, which is the basic semialgebraic set defined by
\begin{align}\label{hexagon_equations}
x,\,
y,\,
3-x,\,
3-y,\,
2xy-3x-3y+6\ge 0.
\end{align}
Since $\mathcal{A}_{\ge 0} \subseteq \mathcal{A}_{\Delta\ge 0}$, it suffices to show that $\rho_V(\Fl_{1,3;4}^{\Delta>0})$ is contained in the curved hexagon and that $\mathcal{A}_{\ge 0}$ is dense in the curved hexagon.

First we verify that $\rho_V(\Fl_{1,3;4}^{\Delta>0})$ is contained in the interior of the curved hexagon. That is, we assume that $(x,y)$ is as in \eqref{hexagon_restriction_formula} and that \eqref{plucker_positive_134} holds; we must then show that the inequalities \eqref{hexagon_equations} hold strictly. We see by inspection that $0 < x,y < 3$. (In fact, $[0,3] \times [0,3]$ is the product $\rho_V(\Gr_{3,4}^{\ge 0}) \times \rho_V(\Gr_{1,4}^{\ge 0})$ of the two corresponding Grassmannian amplituhedra under our choice of coordinates.) It remains to check that $2xy - 3x - 3y + 6 > 0$. We have
\[
2xy - 3x - 3y + 6 = \frac{3f}{(ad-be+c+d+e+1)(3a+3b+c+1)},
\]
where
\[
f = (ad-be+c)(3b+2c) + 2(be-c) + 2ae + 2bd + cd + 3a + e + 2.
\]
We must show that $f > 0$ when \eqref{plucker_positive_134} holds. We see that $f > 0$ if $be - c \ge 0$. To handle the case $c - be > 0$, we rewrite $f$ as
\[
f = 3b(ad-be+c) + 2acd + 2\big((c-be)(c-1)+1\big)+2ae+2bd+cd+3a+e.
\]
Therefore to conclude that $f > 0$ when $c - be > 0$, it suffices to show that $(c-be)(c-1)+1 \ge 0$, which can easily be verified in each of the cases $c \ge 1$ and $0 < c < 1$.

\subsection{Tilings}\label{sec:hexagon_tilings}
In order to show that $\mathcal{A}_{\ge 0}$ is dense in the curved hexagon, it suffices to exhibit a tiling in the sense of \cref{sec:problems_tilings}. In fact, $\mathcal{A}_{\ge 0}$ has several tilings; two different ones are shown in \cref{figure_hexagon_tilings}.
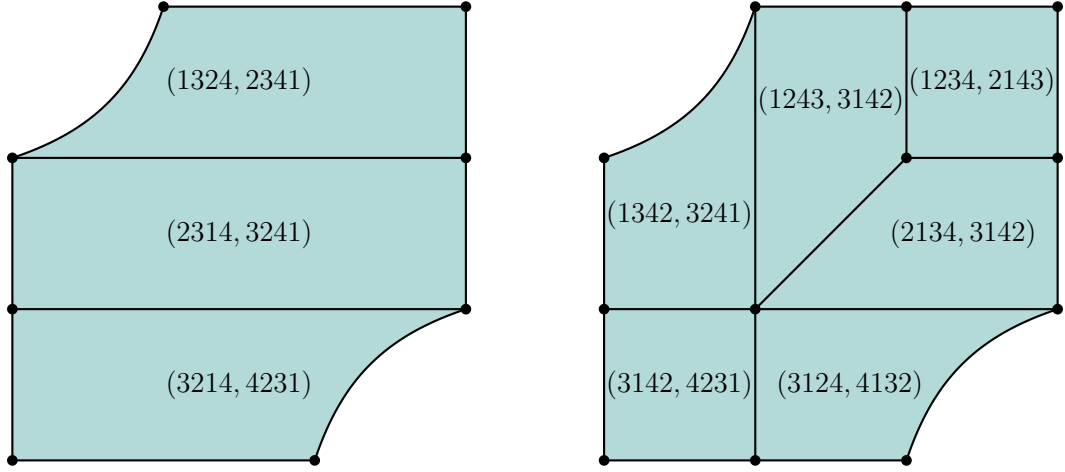
\begin{figure}[tb]
\begin{center}
\[
\begin{tikzpicture}[baseline=(current bounding box.center),scale=2]
\tikzstyle{out1}=[inner sep=0,minimum size=1.2mm,circle,draw=black,fill=black,semithick]
\pgfmathsetmacro{\s}{0.9};
\draw[thick,fill=teal!30](0,0)--(2,0)--plot[smooth,samples=100,domain=2:3](\x,{(6-3*\x)/(3-2*\x)})--(3,1)--(3,3)--(1,3)--plot[smooth,samples=100,domain=1:0](\x,{(6-3*\x)/(3-2*\x)})--(0,2)--cycle;
\draw[thick](0,1)--(3,1);
\draw[thick](0,2)--(3,2);
\node[out1](v1234)at(3,3){};
\node[out1](v2134)at(3,2){};
\node[out1](v3124)at(3,1){};
\node[out1](v1342)at(1,3){};
\node[out1](v4123)at(2,0){};
\node[out1](v2341)at(0,2){};
\node[out1](v3241)at(0,1){};
\node[out1](v4231)at(0,0){};
\node[inner sep=0]at(1.5,2.5){\scalebox{\s}{$(1324,2341)$}};
\node[inner sep=0]at(1.5,1.5){\scalebox{\s}{$(2314,3241)$}};
\node[inner sep=0]at(1.5,0.5){\scalebox{\s}{$(3214,4231)$}};
\end{tikzpicture}
\hspace*{48pt}
\begin{tikzpicture}[baseline=(current bounding box.center),scale=2]
\tikzstyle{out1}=[inner sep=0,minimum size=1.2mm,circle,draw=black,fill=black,semithick]
\pgfmathsetmacro{\s}{0.9};
\draw[thick,fill=teal!30](0,0)--(2,0)--plot[smooth,samples=100,domain=2:3](\x,{(6-3*\x)/(3-2*\x)})--(3,1)--(3,3)--(1,3)--plot[smooth,samples=100,domain=1:0](\x,{(6-3*\x)/(3-2*\x)})--(0,2)--cycle;
\draw[thick](1,0)--(1,1);
\draw[thick](0,1)--(1,1);
\draw[thick](1,1)--(3,1);
\draw[thick](1,1)--(1,3);
\draw[thick](1,1)--(2,2);
\draw[thick](2,2)--(3,2);
\draw[thick](2,2)--(2,3);
\node[out1](v1234)at(3,3){};
\node[out1](v2134)at(3,2){};
\node[out1](v1243)at(2,3){};
\node[out1](v3124)at(3,1){};
\node[out1](v2143)at(2,2){};
\node[out1](v1342)at(1,3){};
\node[out1](v4123)at(2,0){};
\node[out1](v3142)at(1,1){};
\node[out1](v2341)at(0,2){};
\node[out1](v4132)at(1,0){};
\node[out1](v3241)at(0,1){};
\node[out1](v4231)at(0,0){};
\node[inner sep=0]at(2.5,2.5){\scalebox{\s}{$(1234,2143)$}};
\node[inner sep=0]at(2.375,1.5){\scalebox{\s}{$(2134,3142)$}};
\node[inner sep=0]at(1.5,2.375){\scalebox{\s}{$(1243,3142)$}};
\node[inner sep=0]at(1.625,0.5){\scalebox{\s}{$(3124,4132)$}};
\node[inner sep=0]at(0.5,1.625){\scalebox{\s}{$(1342,3241)$}};
\node[inner sep=0]at(0.5,0.5){\scalebox{\s}{$(3142,4231)$}};
\end{tikzpicture}
\]
\caption{Two different tilings of the flag amplituhedron $\rho_V(\Fl_{1,3;4}^{\ge 0})$ from \cref{sec:hexagon} (see \cref{figure_hexagon}). Each tile is the image under $\rho_V$ of a $2$-dimensional cell $\Pi_{v,w}^{>0}$ of $\Fl_{1,3;4}^{\ge 0}$; we have labeled each cell by $(v,w)$.}
\label{figure_hexagon_tilings}
\end{center}
\end{figure}

We explain how to verify that these are valid tilings. Recall from \eqref{cell_decomposition} that the cells of $\Fl_{1,3;4}^{\ge 0}$ are the projections $\Pi_{v,w}^{>0}$ of the totally positive Richardson cells $R_{v,w}^{>0}$ of the complete flag variety $\Fl_{1,2,3;4}^{\ge 0}$; they are indexed by pairs $(v,w)$, where $v\le w$ in $\symgp{4}$ and $w(2) < w(3)$ (i.e.\ $w$ is a minimal-length coset representative of $\symgp{4}/\symgp{\tilde{J}}$).

Let us check that the interior of the curved rectangle at the top of the left tiling in \cref{figure_hexagon_tilings} is $\rho_V(\Pi_{1324,2341}^{>0})$. We calculate $R_{1324,2341}^{>0}$ using the parametrization of Marsh and Rietsch \cite{marsh_rietsch04}. We fix the reduced word $s_1s_2s_3$ for $2341$; the positive distinguished (i.e.\ rightmost) subexpression for $1324$ is then $s_1\underline{s_2}s_3$. To form the Marsh--Rietsch parametrization, we replace each $s_i$ with either $\dot{s}_i$ (if it is underlined) or $y_i(t)$ for a parameter $t > 0$ (otherwise):
\[
y_1(p)\dot{s}_2y_3(q) = 
\scalebox{1}{$\begin{bmatrix}
1 & 0 & 0 & 0 \\
p & 1 & 0 & 0 \\
0 & 0 & 1 & 0 \\
0 & 0 & 0 & 1
\end{bmatrix}$}
\scalebox{1}{$\begin{bmatrix}
1 & 0 & 0 & 0 \\
0 & 0 & -1 & 0 \\
0 & 1 & 0 & 0 \\
0 & 0 & 0 & 1
\end{bmatrix}$}
\scalebox{1}{$\begin{bmatrix}
1 & 0 & 0 & 0 \\
0 & 1 & 0 & 0 \\
0 & 0 & 1 & 0 \\
0 & 0 & q & 1
\end{bmatrix}$} =
\begin{bmatrix}
1 & 0 & 0 & 0 \\
p & 0 & -1 & 0 \\
0 & 1 & 0 & 0 \\
0 & 0 & q & 1
\end{bmatrix}.
\]
When $p,q> 0$, this matrix represents a point of $\Fl_{1,3;4}^{\geq 0}$ whose Pl\"{u}cker coordinates are
\[
\begin{gathered}
\Delta_1 = 1, \enspace
\Delta_2 = p, \enspace
\Delta_3 = 0, \enspace
\Delta_4 = 0, \\
\Delta_{1,2,3} = 1, \enspace
\Delta_{1,2,4} = 0, \enspace
\Delta_{1,3,4} = q, \enspace
\Delta_{2,3,4} = pq.
\end{gathered}
\]
This is obtained from \eqref{hexagon_pluckers} by setting $a=p$, $d=q$, and $b=c=e=0$, so by \eqref{hexagon_restriction_formula}, $\rho_V$ sends the point above to
\[
\left(\frac{q+3}{pq+q+1} \,,\, \frac{6p+3}{3p+1}\right).
\]
For fixed $p > 0$, as $q$ ranges over $\mathbb{R}_{>0}$ this gives the (open) horizontal line segment from $(\frac{1}{p+1},\frac{6p+3}{3p+1})$ (which lies on the hyperbola $(2x-3)(2y-3) = -3$) to $(3,\frac{6p+3}{3p+1})$. Then letting $p$ range over $\mathbb{R}_{>0}$, we see that $\rho_V$ is injective on $\Pi_{1324,2341}^{>0}$ and its image is the interior of the curved rectangle as claimed.

Let us similarly check that the interior of the curved rectangle at the bottom of the left tiling in \cref{figure_hexagon_tilings} is $\rho_V(\Pi_{3214,4231}^{>0})$. We fix the reduced word $s_1s_2s_3s_2s_1$ for $4231$; the positive distinguished subexpression for $3214$ is then $\underline{s_1}s_2s_3\underline{s_2s_1}$. This gives the parametrization
\[
\dot{s}_1y_2(p)y_3(q)\dot{s}_2\dot{s}_1 = 
\scalebox{0.6}{$\begin{bmatrix}
0 & -1 & 0 & 0 \\
1 & 0 & 0 & 0 \\
0 & 0 & 1 & 0 \\
0 & 0 & 0 & 1
\end{bmatrix}$}
\scalebox{0.6}{$\begin{bmatrix}
1 & 0 & 0 & 0 \\
0 & 1 & 0 & 0 \\
0 & p & 1 & 0 \\
0 & 0 & 0 & 1
\end{bmatrix}$}
\scalebox{0.6}{$\begin{bmatrix}
1 & 0 & 0 & 0 \\
0 & 1 & 0 & 0 \\
0 & 0 & 1 & 0 \\
0 & 0 & q & 1
\end{bmatrix}$}
\scalebox{0.6}{$\begin{bmatrix}
1 & 0 & 0 & 0 \\
0 & 0 & -1 & 0 \\
0 & 1 & 0 & 0 \\
0 & 0 & 0 & 1
\end{bmatrix}$}
\scalebox{0.6}{$\begin{bmatrix}
0 & -1 & 0 & 0 \\
1 & 0 & 0 & 0 \\
0 & 0 & 1 & 0 \\
0 & 0 & 0 & 1
\end{bmatrix}$} = 
\begin{bmatrix}
0 & 0 & 1 & 0 \\
0 & -1 & 0 & 0 \\
1 & 0 & -p & 0 \\
q & 0 & 0 & 1
\end{bmatrix}.
\]
When $p,q> 0$, this matrix represents a point of $\Fl_{1,3;4}^{\geq 0}$ whose Pl\"{u}cker coordinates are
\[
\begin{gathered}
\Delta_1 = 0, \enspace
\Delta_2 = 0, \enspace
\Delta_3 = 1, \enspace
\Delta_4 = q, \\
\Delta_{1,2,3} = 1, \enspace
\Delta_{1,2,4} = q, \enspace
\Delta_{1,3,4} = 0, \enspace
\Delta_{2,3,4} = pq.
\end{gathered}
\]
Recall that the Pl\"{u}cker coordinates in each row are defined modulo rescaling. In order to obtain this point from \eqref{hexagon_pluckers}, we set $c = bq$, $d = \frac{pq}{a}$, and $e=q$; then we send $b\to\infty$ before sending $a\to\infty$. By \eqref{hexagon_restriction_formula}, $\rho_V$ sends the point above to
\begin{multline*}
\lim_{a\to\infty}\lim_{b\to\infty}\left(\frac{\frac{pq}{a}+2q+3}{pq+\frac{pq}{a}+q+1} \,,\, \frac{6a+3b+3}{3a+3b+bq+1}\right) \\
=
\lim_{a\to\infty}\left(\frac{\frac{pq}{a}+2q+3}{pq+\frac{pq}{a}+q+1} \,,\, \frac{3}{3+q}\right)
=
\left(\frac{2q+3}{pq+q+1} \,,\, \frac{3}{3+q}\right).
\end{multline*}
For fixed $q > 0$, as $p$ ranges over $\mathbb{R}_{>0}$ this gives the (open) horizontal line segment from $(0,\frac{3}{3+q})$ to $(\frac{2q+3}{q+1},\frac{3}{3+q})$ (and the latter endpoint lies on the hyperbola $(2x-3)(2y-3) = -3$). Then letting $q$ range over $\mathbb{R}_{>0}$, we see that $\rho_V$ is injective on $\Pi_{3214,4231}^{>0}$ and its image is the interior of the curved rectangle as claimed.

We have performed the same calculation for every tile in \cref{figure_hexagon_tilings}, but omit the details since they are similar to the cases above (with $\rho_V(\Pi_{3214,4231}^{>0})$ being the most difficult).

\subsection{Canonical differential form}\label{sec:hexagon_form}
We calculate the canonical differential form of $\mathcal{A}_{\ge 0}$ as a positive geometry in $\PP^1\times\PP^1$. (Our calculations and the fact that $\mathcal{A}_{\ge 0}$ is a positive geometry follow from similar arguments as in work of Telen \cite[Example 1.10]{telen}.)

The canonical form of $\mathcal{A}_{\ge 0}$ is $f(x,y) dx \wedge dy$, where $f(x,y)$ is the canonical rational function. We tile $\mathcal{A}_{\geq 0}$ into $9$ regions, which are the connected components after deleting the $4$ lines $x=1$, $x=2$, $y=1$, $y=2$. Since canonical rational functions are additive (with appropriate signs), it suffices to find the canonical rational function of each region and then sum them all.

Seven of the regions are squares. The canonical rational function of the unit square $[0,1]\times [0,1]$ is
\[
\frac{1}{xy(1-x)(1-y)},
\]
and we obtain the canonical rational functions of the seven squares by translation.

The two remaining regions are curved triangles. It suffices to determine the canonical rational function of the curved triangle with vertices $(0,0)$, $(1,0)$, and $(0,1)$, bounded by the coordinate axes and the hyperbola $(2x+1)(2y+1)=3$. (The two curved triangles in our tiling are each obtained by reflecting over a coordinate axis and then translating.) We claim that its canonical rational function is
\[
\frac{1}{xy(1-x-y-2xy)}.
\]
Indeed, we can check that it has the correct residues on the coordinate axes. To check its residue along the curved edge, write
\[
\frac{1}{xy(1-x-y-2xy)}dx\wedge dy = \frac{-1}{xy(2x+1)}dx \wedge \frac{d(1-x-y-2xy)}{1-x-y-2xy},
\]
and observe that $y(2x+1) = 1-x$ on the curved edge.

Finally, by summing the canonical rational functions of the $9$ regions (using \texttt{Maple}), we get that the canonical rational function of $\mathcal{A}_{\ge 0}$ is
\[
f(x,y) = \frac{6(2xy-3x-3y+9)}{xy(3-x)(3-y)(2xy-3x-3y+6)}.
\]

\bibliographystyle{alpha}
\bibliography{ref}

\end{document}